\documentclass[12pt]{amsart}
\usepackage{amssymb}
\usepackage{enumerate}
\usepackage{hyperref,hyphenat}
\hypersetup{hidelinks}
\usepackage[all]{xy}
\usepackage[toc,page,title,titletoc,header]{appendix}
\usepackage{xcolor}

\numberwithin{equation}{section}

\theoremstyle{plain}
\newtheorem{Thm}{Theorem}[section]
\newtheorem{Prop}[Thm]{Proposition}
\newtheorem{Lem}[Thm]{Lemma}
\newtheorem{Que}[Thm]{Question}
\newtheorem{Cor}[Thm]{Corollary}
\newtheorem{Conj}[Thm]{Conjecture}

\theoremstyle{remark}
\newtheorem{Rem}[Thm]{Remark}

\theoremstyle{definition}
\newtheorem{Def}[Thm]{Definition}
\newtheorem{Eg}[Thm]{Example}

\newcommand{\Mod}[1]{\ (\mathrm{mod}\ #1)}
\newcommand{\Rom}[1]{\uppercase\expandafter{\romannumeral #1\relax}}

\def\A{\mathbb{A}}
\def\Z{\mathbb{Z}}
\def\Q{\mathbb{Q}}
\def\R{\mathbb{R}}
\def\C{\mathbb{C}}
\def\P{\mathbb{P}}
\def\D{\mathbb{D}}
\def\G{\mathbb{G}}

\def\sL{\mathcal{L}}

\def\sA{\mathcal{A}}

\def\sM{\mathcal{M}}

\def\Pr{\mathrm{Pr}}

\def\Spec{\mathrm{Spec}}
\def\PrePer{\mathrm{PrePer}}

\def\Gal{\mathrm{Gal}}
\def\Id{\mathrm{Id}}
\def\supp{\mathrm{supp}}
\def\Berk{\mathrm{Berk}}
\def\Aut{\mathrm{Aut}}
\def\ep{\mathrm{ep}}

\newcommand{\lcm}{\operatorname{lcm}}

\def\la{\lambda}
\def\La{\Lambda}

\begin{document}
\title[Rank-two recurrence for polynomials]{Rank-two recurrence results for polynomials and questions of dynamical Mordell--Lang type}

\author{Geng-Rui Zhang}
\address{School of Mathematical Sciences, Peking University, Beijing 100871, China}
\email{grzhang@stu.pku.edu.cn, chibasei@163.com}

\date{October 6, 2026}



\maketitle

\begin{abstract}
Let $f,g\in\C[z]\setminus\C$ and $c\in\C[z]$. Suppose that $\deg(c)=1$ if $\deg(f)=\deg(g)=1$. Using the theory of Presburger arithmetic, we prove that the rank-two recurrence set
\[S_{f,g,c}^2:=\left\lbrace(m,n)\in\Z_{\geq0}^2\colon \exists\la\in\C, f^{\circ m}(\la)=g^{\circ n}(\la)=c(\la)\right\rbrace\]
is semi-linear. This is a generalization of a theorem of Yang and Zhong for the case $m=n$. We also obtain partial results on recurrence sets for rational maps. These results are related to higher-dimensional questions of dynamical Mordell--Lang type of rank $\leq2$.
\end{abstract}

\tableofcontents

\section{Introduction}

\subsection{A theorem of Yang and Zhong}\label{secYZ}
For every rational map $h\in\C(z)$ and every integer $n\in\Z_{>0}$, the $n$-th iterate of $h$ is denoted by $h^{\circ n}$, and by convention we set $h^{\circ 0}(z)=z$. Recall that an arithmetic progression in $\Z_{\geq0}$ is a subset $Z$ of $\Z_{\geq0}$ of the form
\[Z=\left\lbrace a+bk\colon k\in\Z_{\geq0}\right\rbrace\]
for some $a,b\in\Z_{\geq0}$. If $b=0$, then $Z=\{a\}$ is a singleton. Recently, Yang and Zhong proved the following theorem \cite[Theorem~1.14]{YZ26}:
\begin{Thm}[Yang--Zhong]\label{thmYZ}
	Let $f,g\in\C[z]\setminus\C$ and $c\in\C[z]$. Then the set
	\[S_{f,g,c}:=\left\lbrace n\in\Z_{\geq0}\colon \exists\la\in\C, f^{\circ n}(\la)=g^{\circ n}(\la)=c(\la)\right\rbrace\]
	is a finite union of arithmetic progressions.
\end{Thm}
Here we adopt the convention that the union over an empty index set is empty. The above theorem of Yang and Zhong is motivated by a question of Xie (Question~\ref{Qxie}); see \cite[\S~1.3]{YZ26}.

If we switch the roles of $n$ and $\la$ in the definition of $S_{f,g,c}$, it is also interesting and natural to study the set
\[Q_{f,g,c}:=\left\lbrace \la\in\C\colon \exists n\in\Z_{\geq0}, f^{\circ n}(\la)=g^{\circ n}(\la)=c(\la)\right\rbrace.\]
We remark that this set $Q_{f,g,c}$ and its variants have been widely studied in the literature; see, for example, \cite{HT17,NZ25,YZ26}. Results concerning the finiteness of $Q_{f,g,c}$ are dynamical analogues of a theorem of Bugeaud--Corvaja--Zannier \cite[Theorem~1]{BCZgcd} on greatest common divisors. See \cite[\S~1.1]{YZ26} for more information about the historical background.

\subsection{Our generalization to rank two}\label{secrk2}
The aim of this article is to generalize Theorem~\ref{thmYZ}. There are several possible directions for generalizing Theorem~\ref{thmYZ}.

First, it seems that only dealing with $f^{\circ n}$ and $g^{\circ n}$ for the same $n\in\Z_{\geq0}$ is a considerable restriction. In fact, the proof of Theorem~\ref{thmYZ} in \cite{YZ26} reflects this problem to some extent. In \cite{YZ26}, the proof of Theorem~\ref{thmYZ} (when $\deg(f)\geq2$ and $\deg(g)\geq2$) is divided into two cases:
\begin{enumerate}
	\item \label{degsame} $\deg(f)=\deg(g)$;
	\item \label{degdiff} $\deg(f)\neq\deg(g)$.
\end{enumerate}
We want to generalize Theorem~\ref{thmYZ} and obtain a more unified statement for these cases (\ref{degsame}) and (\ref{degdiff}), by studying rank-two recurrence sets defined as follows.
\begin{Def}
	Let $f,g\in\C[z]\setminus\C$ and $c\in\C[z]$. Define the \emph{rank-two recurrence set of $(f,g,c)$} as
	\[S_{f,g,c}^2:=\left\lbrace(m,n)\in\Z_{\geq0}^2\colon \exists\la\in\C, f^{\circ m}(\la)=g^{\circ n}(\la)=c(\la)\right\rbrace.\]
\end{Def}
The following notions of (semi-)linear sets are the higher-rank analogues of (finite unions of) arithmetic progressions:
\begin{Def}\label{semilinear}
	Let $k\geq1$ be an integer.
	
	(1) A subset $Z$ of $\Z_{\geq0}^k$ is called \emph{linear} if it is of the form
	\[Z=L(a,P):=\left\lbrace a+\sum_{j=1}^r b_j\colon r\in\Z_{\geq0}\text{ and }b_j\in P\text{ for }1\leq j\leq r\right\rbrace\]
	for some $a\in\Z_{\geq0}^k$ and some finite subset $P\subseteq\Z_{\geq0}^k$. We call $P$ a \emph{set of generators} of $L(a,P)$, and $a$ the \emph{initial vector} of $L(a,P)$.
	
	(2) A subset $Z$ of $\Z_{\geq0}^k$ is called \emph{semi-linear} if it is a finite union of linear subsets of $\Z_{\geq0}^k$.
\end{Def}
\begin{Rem}
	(i) Note that $L(a,P)$ is precisely a translate by $a$ of the submonoid of $(\Z_{\geq0}^k,+,0)$ generated by $P$.
	
	(ii) Since any $k+1$ vectors in $\Z_{\geq0}^k$ are linearly dependent over $\Q$, it is easy to see that every linear subset of $\Z_{\geq0}^k$ is a finite union of linear subsets with $\leq k$ generators, and therefore the same holds for semi-linear subsets.
	
	(iii) For the rank-one case ($k=1$), it is easy to see that for all $Z\subseteq\Z_{\geq0}$, $Z$ is semi-linear if and only if $Z$ is a finite union of arithmetic progressions.
\end{Rem}

\smallskip

Our main result is the following theorem, which is a rank-two generalization of Theorem~\ref{thmYZ}:
\begin{Thm}\label{thmrk2poly}
	Let $f,g\in\C[z]\setminus\C$ be such that $\deg(f)\geq2$ or $\deg(g)\geq2$, and let $c\in\C[z]$. Then the rank-two recurrence set
	\[S_{f,g,c}^2=\left\lbrace(m,n)\in\Z_{\geq0}^2\colon \exists\la\in\C, f^{\circ m}(\la)=g^{\circ n}(\la)=c(\la)\right\rbrace\]
	is semi-linear.
\end{Thm}
For polynomial triples $(f,g,c)$ satisfying $\deg(f)\deg(g)\geq2$, Theorem~\ref{thmrk2poly} immediately implies that $S_{f,g,c}^2$ is a finite union of arithmetic progressions; see Corollary~\ref{rk2tork1}. However, the proof of Theorem~\ref{thmrk2poly} relies on a special case of Theorem~\ref{thmYZ}, together with techniques from the proof of Theorem~\ref{thmYZ} in \cite{YZ26}.

The case of Theorem~\ref{thmrk2poly} where exactly one of $f$ and $g$ has degree one is easy. For the case where $\deg(f)\geq2$ and $\deg(g)\geq2$, arguing as in \cite[\S~4]{YZ26}, we show that $f$ and $g$ have the same measure of maximal entropy under certain conditions (see \S~\ref{secsamemupoly}), and then apply a well-known classification \cite{SS95} of such pairs $(f,g)$. The characterization of semi-linear sets using Presburger arithmetic on $\Z_{\geq0}$ (see Theorem~\ref{thmGS}), the torsion points theorem on $\G_{m,\C}^2$, and explicit descriptions of periodic curves under split polynomial endomorphisms \cite{GNY19} are also important for the proof of Theorem~\ref{thmrk2poly}.

\subsubsection*{The case where $\deg(f)=\deg(g)=1$}
When $\deg(f)=\deg(g)=1$, there exist examples for which $S_{f,g,c}^2$ is not semi-linear, which is quite different from the rank-one case (see Theorem~\ref{thmYZ}). Thus, we assume that
\[\deg(f)\geq2\quad\text{or}\quad\deg(g)\geq2\]
in our Theorem~\ref{thmrk2poly}. The following examples illustrate this phenomenon.
\begin{Eg}\label{deg1counter}
(1) Let
\[f(z)=\frac{1}{2} z,\quad g(z)=z-1,\quad\text{and}\quad c(z)=1.\]
We compute that
\[S_{f,g,c}^2=\left\lbrace (m,2^m-1)\colon m\in\Z_{\geq0}\right\rbrace.\]
Suppose that $S_{f,g,c}^2$ is semi-linear. It follows from Corollary~\ref{semiclosed} that
\[\{2^{m}-1\colon m\in\Z_{>0}\}\]
is semi-linear, which is impossible. Thus, $S_{f,g,c}^2$ is not semi-linear.

(2) Let $k$ be an arbitrary integer $\geq2$. Define
\[f(z)=2z,\quad g(z)=z+1,\quad\text{and}\quad c(z)=z^k.\]
A direct computation shows that
\[S_{f,g,c}^2=\left(\Z_{\geq0}\times\{0\}\right)\bigsqcup\left\lbrace (r(k-1),2^{rk}-2^r)\colon r\in\Z_{>0}\right\rbrace.\]
If $S_{f,g,c}^2$ is semi-linear, then Corollary~\ref{semiclosed} shows that
\[\{2^{rk}-2^r\colon r\in\Z_{>0}\}\]
is also semi-linear, which is a contradiction. Hence $S_{f,g,c}^2$ is not semi-linear.
\end{Eg}

It is interesting to further classify all polynomial triples $(f,g,c)$ with
\[\deg(f)=\deg(g)=1\]
such that $S_{f,g,c}^2$ is not semi-linear. We are able to prove the following result:
\begin{Thm}\label{deg111}
	Let $f,g,c\in\C[z]$ be polynomials of degree $1$. Then $S_{f,g,c}^2$ is semi-linear.
\end{Thm}
The proof of Theorem~\ref{deg111} relies on the Mordell--Lang conjecture on $\G_{m,\C}^2$ proved by Laurent \cite{Laurent}, and a special case of the following proposition:
\begin{Prop}[= Proposition~\ref{infUCg}]\label{infUCintro}
	Let $R\in\C(z)\setminus\C$ be a non-constant rational map, and let
	\[\eta\in\partial\D=\left\lbrace z\in\C\colon\left|z\right|=1\right\rbrace.\]
	Suppose that there exists $\gamma\in\C$ such that
	\[\#\left\lbrace n\in\Z\colon R(\eta^n)\in\gamma\cdot\Z\right\rbrace=\infty.\]
	Then $\eta$ is a root of unity.
\end{Prop}
Proposition~\ref{infUCintro} may be viewed as a rigidity statement for integral values of rational maps along multiplicative subgroups of the unit circle.
\subsubsection*{Modified recurrence sets for power maps}
Let 
\[f(z)=\zeta_1 z^{d_1},\quad g(z)=\zeta_2 z^{d_2},\quad\text{and}\quad c(z)=\zeta_3 z^{d_3},\]
where $d_1,d_2\geq2$ and $d_3\geq1$ are integers, and $\zeta_1,\zeta_2,\zeta_3$ are roots of unity. Observe that $f(0)=g(0)=c(0)=0$. The choice $\la=0$ shows that
\[S_{f,g,c}=\Z_{\geq0}\quad\text{and}\quad S_{f,g,c}^2=\Z_{\geq0}^2,\]
which are semi-linear. After excluding the trivial solution $\la=0$, Yang and Zhong \cite[Corollary~4.4]{YZ26} proved that the \emph{modified rank-one recurrence set}
\[\tilde{S}_{f,g,c}:=\left\lbrace n\in\Z_{\geq0}\colon \exists\la\in\C^\times,\,f^{\circ n}(\la)=g^{\circ n}(\la)=c(\la)\right\rbrace\]
is also semi-linear. In fact, they proved a slightly stronger version. However, the rank-two case behaves very differently, as shown by the following proposition:
\begin{Prop}\label{powertil}
	Let $r,s\geq3$ be odd integers. Define
	\[f(z)=z^r,\quad g(z)=z^s,\quad\text{and}\quad c(z)=-z.\]
	Define the \emph{modified rank-two recurrence set} to be
	\[\tilde{S}_{f,g,c}^2:=\left\lbrace (m,n)\in\Z_{\geq0}^2\colon \exists\la\in\C^\times,\,f^{\circ m}(\la)=g^{\circ n}(\la)=c(\la)\right\rbrace.\]
	Then $\tilde{S}_{f,g,c}^2$ is not semi-linear.
\end{Prop}
The methods used in the proof of Proposition~\ref{powertil} can produce additional examples. It is natural to investigate further the modified rank-two recurrence sets for power maps.

\subsection{Partial results for rational maps}\label{secrat}
Besides considering rank-two recurrence sets as in \S~\ref{secrk2}, another natural direction to generalize Theorem~\ref{thmYZ} is to study the case of rational maps $f,g,c\in\C(z)$. We define rational map versions of recurrence sets as follows:
\begin{Def}
	Let $f,g\in\C(z)\setminus\C$ and $c\in\C(z)$.
	\begin{enumerate}
		\item Define the \emph{rank-one recurrence set of $(f,g,c)$} as
		\[\hat{S}_{f,g,c}:=\left\lbrace n\in\Z_{\geq0}\colon\exists\la\in\P^1(\C),f^{\circ n}(\la)=g^{\circ n}(\la)=c(\la)\right\rbrace.\]
		\item Define the \emph{rank-two recurrence set of $(f,g,c)$} as
		\[\hat{S}_{f,g,c}^2:=\left\lbrace(m,n)\in\Z_{\geq0}^2\colon\exists\la\in\P^1(\C),f^{\circ m}(\la)=g^{\circ n}(\la)=c(\la)\right\rbrace.\]
	\end{enumerate}
\end{Def}

Note that if $f$, $g$, and $c$ are all non-constant polynomials, then the existence of the common fixed point $\infty\in\P^1(\C)$ shows that
\[\hat{S}_{f,g,c}=\Z_{\geq0}\quad\text{and}\quad\hat{S}_{f,g,c}^2=\Z_{\geq0}^2.\]

\medskip

In complex dynamics, the exceptional maps defined below are often regarded as special classes of rational maps of degree $\geq2$. They are the rational maps on $\P^1_\C$ that are related to algebraic groups and exhibit special dynamical properties. For endomorphisms $g:X\to X$ and $h:Y\to Y$ of algebraic varieties, we say that $h$ is \emph{semi-conjugate} to $g$ if there is a dominant morphism $\pi:X\to Y$ such that $\pi\circ g=h\circ\pi$, in which case we write $g\geq h$ (or $g\geq_\pi h$ when $\pi$ is specified).
\begin{Def}\label{excep}
	Let $f:\P^1_\C\to\P^1_\C$ be an endomorphism of degree $d\geq2$. 
	\begin{itemize}
		\item The map $f$ is called a \emph{Latt\`es map} if there exists an endomorphism $\phi$ of an elliptic curve $E$ such that $\phi\geq f$. The map $f$ is called \emph{flexible Latt\`es} if there exist an elliptic curve $E$, an integer $n\in\Z\setminus\{0,\pm1\}$, and a finite morphism $\pi:E\to\P^1$ of degree two such that $[n]\geq_\pi f$, where $[n]$ denotes the multiplication-by-$n$ map on $E$. A non-flexible Latt\`es map is called \emph{rigid Latt\`es}.
		\item We say that $f$ is of \emph{monomial type} if it is semi-conjugate to a power map on $\P^1$. It is well known that $f$ is of monomial type if and only if it is conjugate to $z^{\pm d}$ or $\pm T_d(z)$, where $T_d(z)$ is the (normalized) \emph{Chebyshev polynomial of degree $d$}, i.e., the unique monic polynomial $T_d\in\C[z]$ of degree $d$ such that
		\[T_d(z+z^{-1})=z^d+z^{-d}.\]
		See \cite[Chapter~6]{Silverman2007} and \cite{milnor2006lattes}. 
		\item $f$ is called \emph{exceptional} if it is Latt\`es or of monomial type.
	\end{itemize}
	The above notions depend only on the conjugacy class of $f$. Moreover, it is well known that $f$ is exceptional if and only if $f^{\circ k}$ is exceptional for some (equivalently, every) $k\in\Z_{>0}$.
\end{Def}
Latt\`es maps are not (conjugate to) polynomials because their Julia sets are $\P^1(\C)$. Let $f\in\C[z]$ be a polynomial of degree $d\geq2$. Then $f$ is exceptional if and only if $f$ is (affinely) conjugate to $z^d$ or $\pm T_d(z)$.

\medskip

We obtain the following partial result for recurrence sets of rational maps, with emphasis on the non-exceptional case:
\begin{Thm}\label{thmratg}
Let $f,g,c\in\C(z)$ be such that $\deg(f)\geq2$ and $\deg(g)\geq2$.
\begin{enumerate}
	\item \label{rat1g} If $\deg(f)=\deg(g)$ and $f^{\circ (l+k)}=f^{\circ l}\circ g^{\circ k}$ for some integers $l\geq0$ and $k\geq1$, then $\hat{S}_{f,g,c}$ is semi-linear.
	\item \label{rat2g} If $\deg(f)=\deg(g)$ and $f$ is non-exceptional, then $\hat{S}_{f,g,c}$ is semi-linear.
	\item \label{rat3g} If $c\in\C$ and $f$ is non-exceptional, then $\hat{S}_{f,g,c}^2$ is semi-linear; in particular, $\hat{S}_{f,g,c}$ is also semi-linear.
\end{enumerate}
\end{Thm}
As in the polynomial case, we show that $f$ and $g$ have the same measure of maximal entropy under certain conditions (see \S~\ref{secsamemurat}), and then apply a well-known description for such pairs $(f,g)$ (see Theorem~\ref{ratsamemu}). We remark that the case (\ref{rat2g}) of Theorem~\ref{thmratg} can be deduced from the case (\ref{rat1g}); see Corollary~\ref{rat1to2}.

We also prove that $\hat{S}_{f,g,c}^2$ is semi-linear for a special case of exceptional maps in Proposition~\ref{powerconst}.

We pose the following question regarding recurrence sets of rational maps:
\begin{Que}
Let $f,g \in \C(z) \setminus \C$ with $\deg(f) \geq 2$ or $\deg(g) \geq 2$, and let $c \in \C(z)$.
\begin{enumerate}
\item \label{ratq1} Is $\hat{S}_{f,g,c}$ a semi-linear subset of $\Z_{\geq0}$?
\item \label{ratq2} Under what conditions is $\hat{S}_{f,g,c}^2$ a semi-linear subset of $\Z_{\geq0}^2$?
\end{enumerate}
\end{Que}
We have found examples of triples $(f,g,c)$ for which $\hat{S}_{f,g,c}^2$ is not semi-linear. In fact, examples where $\hat{S}_{f,g,c}^2$ fails to be semi-linear can be found in each of the three cases: where $f$ is non-exceptional, a Latt\`es map, or of monomial type. These examples will appear in a forthcoming paper.

\subsection{Questions of dynamical Mordell--Lang type}\label{secDMLQ}
\subsubsection*{Dynamical Mordell--Lang conjecture and higher-dimensional questions}
The dynamical Mordell--Lang conjecture, proposed by Ghioca and Tucker \cite{GT09}, is a dynamical analogue of the classical Mordell--Lang conjecture on subvarieties of semi-abelian varieties:
\begin{Conj}[Dynamical Mordell--Lang conjecture (DML)]
	Let $k$ be an algebraically closed field of characteristic zero. Let $X$ be a quasi-projective variety defined over $k$ and $f:X\dashrightarrow X$ be a dominant rational self-map. Then for every Zariski closed subset $V\subseteq X$ and every point $x\in X(k)$ whose forward $f$-orbit is well-defined, the set
	\[\left\lbrace n\in\Z_{\geq0}\colon f^{\circ n}(x)\in V\right\rbrace\]
	is a finite union of arithmetic progressions.
\end{Conj}
We remark that the original version of DML only deals with endomorphisms. For some progress on DML, see \cite{DMLbook,xDML} and references therein. By replacing the point $x$ with a subvariety $Z$, Xie proposed the following higher-dimensional DML question \cite[Question~69~(i)]{xDML} for endomorphisms:
\begin{Que}[Xie]\label{Qxie}
	Let $k$ be an algebraically closed field of characteristic zero (for example, $\C$). Let $X$ be a quasi-projective variety defined over $k$, and let $f:X\to X$ be an endomorphism. Let $Z$ and $V$ be irreducible subvarieties of $X$. Can one describe the set $\left\lbrace n\in\Z_{\geq0}\colon f^{\circ n}(Z)\cap V\neq\emptyset\right\rbrace$?
\end{Que}
In the original question, the following additional condition on dimensions is imposed:
\begin{equation}\label{dimcon}
	\dim Z+\dim V<\dim X.
\end{equation}
However, the condition \eqref{dimcon} is redundant because it is always satisfied after replacing $(X,f,Z,V)$ with $(X\times\A^N_k,f\times\Id_{\A^N_k},Z\times\{0\},V\times\{0\})$ for any sufficiently large integer $N\gg1$, where $0=(0,\dots,0)\in\A^N_k(k)$.

\medskip

Unfortunately, the higher-dimensional (rank-one) recurrence set
\[\left\lbrace n\in\Z_{\geq0}\colon f^{\circ n}(Z)\cap V\neq\emptyset\right\rbrace\]
may be intricate, and it is not a finite union of arithmetic progressions in general. An elementary example of Lee and Nam (see \cite[Example~1.12]{YZ26}) shows that for $X=\A^5_k$, the set of all composite numbers can be achieved as such a higher-dimensional recurrence set.

On the other hand, the theorem of Yang and Zhong (Theorem~\ref{thmYZ}) can be interpreted in the framework of Question~\ref{Qxie}; hence, it is possible to describe the higher-dimensional recurrence sets in certain special situations. For the reader's convenience, we repeat the interpretation in \cite[\S~1.3]{YZ26} by Yang--Zhong as follows. Let $f,g\in\C[z]\setminus\C$ and $c\in\C[z]$. Set $X=\A^3_\C$. Let $F:X\to X$ be the endomorphism given by
\[F=(f,g,\Id):X\to X,\,(x,y,z)\mapsto(f(x),g(y),z).\]
Define
\[Z=\{(x,x,c(x))\colon x\in\A^1_\C\}\quad\text{and}\quad V=\{(x,x,x)\colon x\in\A^1_\C\},\]
which are irreducible subvarieties of $X$ of dimension one. Then the rank-one recurrence set
\[S_{f,g,c}:=\left\lbrace n\in\Z_{\geq0}\colon \exists\la\in\C, f^{\circ n}(\la)=g^{\circ n}(\la)=c(\la)\right\rbrace\]
has an equivalent description:
\[S_{f,g,c}=\left\lbrace n\in\Z_{\geq0}\colon f^{\circ n}(Z)\cap V\neq\emptyset\right\rbrace.\]
The rational map version $\hat{S}_{f,g,c}$ admits a similar description by replacing $X=\A^3_\C$ and $\A^1_\C$ with $X=(\P^1_\C)^3$ and $\P^1_\C$, respectively.

\subsubsection*{Questions of DML type of higher rank}
Inspired by Question~\ref{Qxie}, we propose the following questions of (higher-dimensional) DML type of higher rank:
\begin{Que}\label{Qhighrk}
	Let $k$ be an algebraically closed field of characteristic zero (for example, $\C$). Let $X$ be a quasi-projective variety defined over $k$, and let $V\subseteq X$ be an irreducible subvariety. Let $r\in\Z_{>0}$ and $f_1,\dots,f_r:X\to X$ be endomorphisms.
	\begin{enumerate}
		\item[(i)] Suppose that the endomorphisms $f_1,\dots,f_r$ are pairwise commuting. Let $Z\subseteq X$ be an irreducible subvariety. Define
		\[R:=\left\lbrace(n_1,\dots,n_r)\in\Z_{\geq0}^r\colon \left(f_1^{\circ n_1}\circ\dots\circ f_r^{\circ n_r}\right)(Z)\cap V\neq\emptyset\right\rbrace.\]
		Can one describe the rank-$r$ recurrence set $R$?
		\item[(ii)] Let $Z_1,\dots,Z_r\subseteq X$ be irreducible subvarieties. Define
		\[R^\prime:=\left\lbrace (n_1,\dots,n_r)\in\Z_{\geq0}^r\colon f_1^{\circ n_1}(Z_1)\cap\dots\cap f_r^{\circ n_r}(Z_r)\cap V\neq\emptyset\right\rbrace.\]
		Can one describe the rank-$r$ recurrence set $R^\prime$?
	\end{enumerate}
\end{Que}
Similarly, the sets $S_{f,g,c}^2$ and $\hat{S}_{f,g,c}^2$ can be interpreted within the framework of Question~\ref{Qhighrk} with $r=2$ as follows. Let $f,g,c\in\C[z]$ be polynomials such that both $f$ and $g$ are non-constant. Set $X=\A^3_\C$. Let $F:X\to X$ be the endomorphism given by
\[F=(f,g,\Id):X\to X,\,(x,y,z)\mapsto(f(x),g(y),z).\]
Define
\[Z=\{(x,x,c(x))\colon x\in\A^1_\C\},\quad V=\{(x,x,x)\colon x\in\A^1_\C\},\]
and
\[V^\prime=\{(x,y,x)\colon x,y\in\A^1_\C\}.\]
Both $Z$ and $V^\prime$ are irreducible subvarieties of $X$. Define three endomorphisms $F,G,G^\prime:X\to X$ by
\[F=(f,\Id,\Id),\quad G=(\Id,g,\Id),\quad\text{and}\quad G^\prime=(g,\Id,\Id).\]
Note that $F$ and $G$ commute: $F\circ G=G\circ F$. Then the rank-two recurrence set
\[S_{f,g,c}^2:=\left\lbrace (m,n)\in\Z_{\geq0}^2\colon \exists\la\in\C, f^{\circ m}(\la)=g^{\circ n}(\la)=c(\la)\right\rbrace\]
has two equivalent descriptions:
\begin{align*}
	S_{f,g,c}^2&=\left\lbrace (m,n)\in\Z_{\geq0}^2\colon \left(F^{\circ m}\circ G^{\circ n}\right)(Z)\cap V\neq\emptyset\right\rbrace\\
	&=\left\lbrace (m,n)\in\Z_{\geq0}^2\colon F^{\circ m}(Z)\cap (G^\prime)^{\circ n}(Z)\cap V^\prime\neq\emptyset\right\rbrace.
\end{align*}
The rational map version $\hat{S}_{f,g,c}^2$ admits analogous descriptions by replacing $X=\A^3_\C$ and $\A^1_\C$ with $X=(\P^1_\C)^3$ and $\P^1_\C$, respectively.

In general, the higher-rank recurrence sets $R$ and $R^\prime$ in Question~\ref{Qhighrk} seem to be extremely intricate and elusive. See Example~\ref{deg1counter} and \cite[\S~2.1]{xDML}. However, we still expect that one may obtain meaningful descriptions in special situations, as illustrated by Theorem~\ref{thmrk2poly}.

\subsection*{Structure of the paper}
In \S~\ref{secPresb}, we recall the theorem of Ginsburg and Spanier that semi-linear sets are precisely definable sets in the Presburger arithmetic structure on $\Z_{\geq0}$, and related results.

In \S~\ref{secsamemu}, assuming $\deg(f)\geq2$ and $\deg(g)\geq2$, we show that $f$ and $g$ have the same measure of maximal entropy under certain conditions, for polynomials in \S~\ref{secsamemupoly} and for rational maps in \S~\ref{secsamemurat}.

In \S~\ref{secexcep}, we analyze the rank-two recurrence sets for exceptional polynomials, and complete the proof of Proposition~\ref{powertil}. We then complete the proofs of Theorem~\ref{thmrk2poly} and Theorem~\ref{deg111} in \S~\ref{secrk2poly} and \S~\ref{sec111}, respectively. The proof of Theorem~\ref{thmratg} is given in \S~\ref{secratpf}.

\section{Presburger arithmetic and semi-linear sets}\label{secPresb}
A theorem of Ginsburg and Spanier asserts that semi-linear sets (see Definition~\ref{semilinear}) are exactly the definable sets in a certain first-order structure $\Pr$ with underlying set $\Z_{\geq0}$. This characterization of semi-linear sets is useful for the proofs of our main theorems. For basic material on first-order model theory, we refer the reader to \cite{217}.

\begin{Def}[Presburger arithmetic]
	The \emph{(modified) Presburger arithmetic} is the first-order structure
	\[\Pr=\left(\Z_{\geq0},+,<,0,1,(P_n)_{n\in\Z_{\geq2}}\right),\]
	where $\left(\Z_{\geq0},+,<,0,1\right)$ have their usual meanings, and for each integer $n\geq2$, $P_n$ is the unary relation consisting of elements divisible by $n$, i.e., $P_n$ is the subset $n\Z_{\geq0}$ of $\Z_{\geq0}$.
\end{Def}
\begin{Rem}
	The language of $\Pr$ is $\sL=\left\lbrace+,<,0,1,(P_n)_{n\in\Z_{\geq2}}\right\rbrace$. Note that $\left(\Z_{\geq0},+,<,0\right)$ is the usual ordered additive monoid structure on $\Z_{\geq0}$.
	
	In the literature, Presburger arithmetic usually refers to the structure
	\[\Pr_{\Z}=\left(\Z,+,<,0,1,(P_{n,\Z})_{n\in\Z_{\geq2}}\right)\]
	(or its theory), where $P_{n,\Z}$ is the unary relation given by $P_{n,\Z}=n\Z$ for each integer $n\geq2$. In this article, we are only concerned with the modified one $\Pr$ (i.e., the restriction of $\Pr_\Z$ to $\Z_{\geq0}$) and call it \emph{Presburger arithmetic} for simplicity.
	
	It is easy to see that the theory of $\Pr$ (also, the theory of $\Pr_\Z$) is a complete decidable theory with quantifier elimination (cf.~\cite[Corollary~3.1.21]{217}). 
\end{Rem}
Ginsburg and Spanier proved that semi-linear sets coincide with the definable sets of $\Pr$:
\begin{Thm}[{\cite[Theorem~1.3]{GS66}}]\label{thmGS}
	For every integer $k\geq1$ and every subset $Z$ of $\Z_{\geq0}^k$, $Z$ is semi-linear in the sense of Definition~\ref{semilinear} if and only if $Z$ is definable in the structure $\Pr$.
\end{Thm}
\begin{Rem}
	By \cite[Exercise~1.4.15]{217}, the definable sets of $\Pr$ (equivalently, semi-linear sets) coincide with the definable sets of the structure
	\[(\Z_{\geq0},+,0,1).\]
	See \cite[\S~1]{GS66} for a description of definable sets of $\Pr$.
\end{Rem}

Theorem~\ref{thmGS} immediately implies the following corollary (see \cite[Theorem~1.1]{GS66}):
\begin{Cor}\label{semiclosed}
	The family of semi-linear subsets of $\Z_{\geq0}^k$ ($k\geq1$) is closed under finite unions, finite intersections, and complementation. Projections of semi-linear sets are semi-linear.
\end{Cor}
As a direct consequence of Corollary~\ref{semiclosed}, we obtain:
\begin{Cor}\label{rk2tork1}
	Let $Z\subseteq\Z_{\geq0}^2$ be semi-linear. Define
	\[\pi(Z):=\left\lbrace n\in\Z_{\geq0}\colon (n,n)\in Z\right\rbrace\subseteq\Z_{\geq0}.\]
	Then $\pi(Z)$ is a finite union of arithmetic progressions.
\end{Cor}
Note that Corollaries~\ref{semiclosed}~and~\ref{rk2tork1} can also be proved by elementary arguments.

\begin{Eg}
	Define $Z=\left\lbrace (m,n)\in\Z_{\geq0}^2\colon m\leq n^2\right\rbrace$. Then $Z$ is a submonoid of the additive monoid $(\Z_{\geq0}^2,+,0)$, and it is not a semi-linear set; see \cite[pp.\,286--287]{GS66}. However, for all integers $a,b\geq0$ and $p,q\geq1$, the set
	\[\left\lbrace n\in\Z_{\geq0}\colon (a+pn,b+qn)\in Z\right\rbrace\]
	is semi-linear. In particular, the converse of Corollary~\ref{rk2tork1} is false.
\end{Eg}

The following definition and lemma are useful in the proof of Proposition~\ref{powertil}.
\begin{Def}
	Let $Z\subseteq\Z_{\geq0}$ be a semi-linear set. Define the \emph{eventual period} $\ep(Z)$ of $Z$ as follows. If $Z$ is finite, then $\ep(Z):=0$. If $Z$ is infinite, then $\ep(Z)$ is the minimal positive integer $m$ such that 
	\[Z\cap [N,\infty)=\left\lbrace N+r+mk\colon r\in R\text{ and }k\in\Z_{\geq0}\right\rbrace\]
	for some integer $N\geq1$ and some $\emptyset\neq R\subseteq\{0,1,\dots,m-1\}$.
\end{Def}

\begin{Lem}\label{uniD}
	Let $Z\subseteq\Z_{\geq0}^2$ be a semi-linear set. For each $m\in\Z_{\geq0}$, set
	\[Z_m:=\left\lbrace n\in\Z_{\geq0}\colon (m,n)\in Z\right\rbrace,\]
	which is semi-linear by Theorem~\ref{thmGS}. Then there exists an integer $D=D(Z)\geq1$ such that for every $m\in\Z_{\geq0}$, either $\ep(Z_m)=0$ or $\ep(Z_m)\mid D$. 
\end{Lem}
\begin{proof}
	Write $Z=\bigcup_{j=1}^N L^{(j)}$, where $N\in\Z_{\geq0}$ and for each $1\leq j\leq N$,
	\[L^{(j)}=L\left(a_j,\{b_{j,i}\}_{i\in I_j}\right)\]
	is a linear subset ($a_j,b_{j,i}\in\Z_{\geq0}^2$ and $I_j$ is a finite index set). Then $Z_m=\bigcup_{j=1}^N L^{(j)}_m$ for all $m\geq0$.

	Fix $1\leq j\leq N$. For $i\in I_j$, write $b_{j,i}=\left(x_{j,i},y_{j,i}\right)$. Define \[I_j^\prime:=\{i\in I_j\colon x_{j,i}=0\text{ and }y_{j,i}>0\}.\]
	If $I_j^\prime=\emptyset$, then $\# L^{(j)}_m<\infty$ for all $m\geq0$. If $I_j^\prime\neq\emptyset$, then for every $m\in\Z_{\geq0}$, either $\# L^{(j)}_m<\infty$, or $\ep(L^{(j)}_m)$ divides
	\[D_j:=\gcd\left\lbrace y_{j,i}\colon i\in I_j^\prime\right\rbrace\in\Z_{>0}.\]
	Set $J=\left\lbrace 1\leq j\leq N\colon I_j^\prime\neq\emptyset\right\rbrace$. Define
	\[D:=\lcm\left\lbrace D_j\colon j\in J\right\rbrace\in\Z_{>0}\]
	if $J\neq\emptyset$, and define $D:=1$ if $J=\emptyset$. It is straightforward to verify that this $D$ satisfies the desired conclusion.
\end{proof}
\begin{Rem}
	Lemma~\ref{uniD} implies that the definable family $(Z_m)_{m\in\Z_{\geq0}}$ admits a uniform upper bound $D$ on the eventual periods.
\end{Rem}

\section{Coincidence of the measures of maximal entropy}\label{secsamemu}
Let $f\in\C(z)$ be a rational map of degree $d\geq2$. The unique \emph{measure of maximal entropy} of $f$ on $\P^1(\C)$ is denoted by $\mu_f$; see \cite{Lyubich83,FLM83,Mane83} for details. The support of $\mu_f$ coincides with the \emph{Julia set} $J(f)\subseteq\P^1(\C)$ of $f$. If $f$ is a polynomial, then $J(f)=\supp(\mu_f)$ is a compact subset of $\C$ and $\mu_f$ is uniquely determined by $J(f)$.

Let $f,g$ be two polynomials (or rational maps) of degree $\geq2$. In this section, we show that $\mu_f=\mu_g$ under suitable conditions, following an argument due to Yang--Zhong \cite[Proof of Proposition~4.5]{YZ26}. Then we recall classification theorems on pairs $(f,g)$ with $\mu_f=\mu_g$, for polynomials and for rational maps. These results are important ingredients for our proofs of Theorems~\ref{thmrk2poly}~and~\ref{thmratg}.

\subsection{The case of polynomials}\label{secsamemupoly}
The following lemma allows us to only consider ``sufficiently large $m$ and $n$''.
\begin{Lem}\label{lemSs}
	Let $f,g,c\in\C[z]$ be polynomials such that $d_1=\deg(f)\geq2$ and $d_2=\deg(g)\geq2$. Let $M_0,N_0\subseteq\Z_{\geq0}$ be finite sets. Then the set
	\[S_s:=\left\lbrace (m,n)\in S_{f,g,c}^2\colon m\in M_0\text{ or }n\in N_0\right\rbrace\]
	is semi-linear.
\end{Lem}
\begin{proof}
	Define $S_s^f=S_{f,g,c}^2\cap\left(M_0\times\Z_{\geq0}\right)$. For $m\in M_0$, set \[S^f_{m}=S_{f,g,c}^2\cap\left(\{m\}\times\Z_{\geq0}\right).\]
	Then $S_s^f=\bigcup_{m\in M_0} S^f_{m}$. Fix an arbitrary $m_0\in M_0$. We prove that $S^f_{m_0}$ is semi-linear by considering the following two cases.
	
	\smallskip
	
	\textbf{(1)} Assume that $f^{\circ m_0}=c$ in $\C[z]$. Since the equation $g^{\circ n}(\la)=c(\la)$ always has a solution in $\C$ whenever $d_2^n\neq\deg(c)$, we conclude that 
	\[S^f_{m_0}=\{(m_0,n)\colon n\in\Z_{\geq0}\text{ such that }\exists\la\in\C,\,g^{\circ n}(\la)=c(\la)\}\]
	equals $\{m_0\}\times\Z_{\geq0}$ or $\{m_0\}\times\left(\Z_{\geq0}\setminus\{\log_{d_2}\deg(c)\}\right)$. (Here we set $\deg(0)=0$ and $\log_{d_2} 0=-\infty$.) Hence $S^f_{m_0}$ is semi-linear.
	
	\smallskip
	
	\textbf{(2)} Assume that $f^{\circ m_0}\neq c$ in $\C[z]$. Then
	\[\La:=\left\lbrace \la\in\C\colon f^{\circ m_0}(\la)=c(\la)\right\rbrace\]
	is finite (possibly empty). For every $\la\in\C$ and every polynomial $h\in\C[z]$, we claim that
	\begin{equation}\label{Shla}
		S_{h,c}(\la):=\left\lbrace m\in\Z_{\geq0}\colon h^{\circ m}(\la)=c(\la)\right\rbrace\text{ is semi-linear.}
	\end{equation}
	If $c(\la)$ is not in the forward $h$-orbit $O_h(\la)$ of $\la$, then $S_{h,c}(\la)=\emptyset$. If $c(\la)\in O_h(\la)$, then $S_{h,c}(\la)\subseteq\Z_{\geq0}$ is an arithmetic progression by considering whether $\la$ is preperiodic for $h$. Thus, $S_{h,c}(\la)$ is semi-linear for all $\la\in\C$ and all $h\in\C[z]$. In particular, we deduce that
	\[S^f_{m_0}=\{m_0\}\times\left(\bigcup_{\la_0\in\La}S_{g,c}(\la_0)\right)\]
	is semi-linear.
	
	\smallskip
	
	We conclude that $S_s^f=\bigcup_{m\in M_0} S^f_{m}$ is semi-linear. Similarly, we obtain that $S_s^g:=S_{f,g,c}^2\cap\left(\Z_{\geq0}\times N_0\right)$ is semi-linear. Therefore, $S_s=S_s^f\cup S_s^g$ is semi-linear.
\end{proof}

\begin{Def}\label{Q2poly}
	Let $f,g,c\in\C[z]$ be polynomials such that $\deg(f)\geq2$ and $\deg(g)\geq2$. Define
	\[Q_{f,g,c}^2:=\left\lbrace \la\in\C\colon \exists (m,n)\in\Z_{\geq0}^2, f^{\circ m}(\la)=g^{\circ n}(\la)=c(\la),\,f^{\circ m}\neq c,\,g^{\circ n}\neq c\right\rbrace.\]
\end{Def}
Note that there exists at most one $m\in\Z_{\geq0}$ such that $f^{\circ m}=c$ (in $\C[z]$) because $\deg(f)\geq2$. Similarly, there exists at most one $n\in\Z_{\geq0}$ such that $g^{\circ n}=c$.

First, we make a reduction by the following observation, which is essentially due to Yang and Zhong \cite[Proof of Theorem~4.8]{YZ26}:
\begin{Prop}\label{Q2polyfinite}
	Let $f,g,c\in\C[z]$ be polynomials such that $\deg(f)\geq2$ and $\deg(g)\geq2$. If $\# Q_{f,g,c}^2<\infty$, then $S_{f,g,c}^2$ is semi-linear.
\end{Prop}
\begin{proof}
	Recall that
	\[S_{f,g,c}^2:=\left\lbrace(m,n)\in\Z_{\geq0}^2\colon \exists\la\in\C, f^{\circ m}(\la)=g^{\circ n}(\la)=c(\la)\right\rbrace.\]
	For every $\la\in Q_{f,g,c}^2$, define
	\[S^2(\la):=\left\lbrace (m,n)\in\Z_{\geq0}^2\colon f^{\circ m}(\la)=g^{\circ n}(\la)=c(\la)\right\rbrace.\]
	Set $A_f=\left\lbrace m\in\Z_{\geq0}\colon f^{\circ m}=c\right\rbrace$ and $A_g=\left\lbrace n\in\Z_{\geq0}\colon g^{\circ n}=c\right\rbrace$. Then
	\[\# A_f\leq1\quad\text{and}\quad\# A_g\leq1.\]
	Note that we have
	\begin{equation}\label{S2decomp}
		S_{f,g,c}^2=S_s\cup\bigcup_{\la\in Q_{f,g,c}^2}S^2(\la),
	\end{equation}
	where
	\[S_s=\left\lbrace (m,n)\in S_{f,g,c}^2\colon m\in A_f\text{ or }n\in A_g\right\rbrace.\]
	
	By Lemma~\ref{lemSs}, $S_s$ is semi-linear. For $\la_0\in Q_{f,g,c}^2$, we have that
	\[S^2(\la_0)=S_{f,c}(\la_0)\times S_{g,c}(\la_0)\]
	is semi-linear by \eqref{Shla} and Theorem~\ref{thmGS}. From the decomposition \eqref{S2decomp} and the assumption that $\# Q_{f,g,c}^2<\infty$, we conclude that $S_{f,g,c}^2$ is semi-linear.
\end{proof}

\begin{Rem}
	We add the conditions $f^{\circ m}\neq c$ and $g^{\circ n}\neq c$ in the definition of $Q_{f,g,c}^2$ in order to obtain Proposition~\ref{poly2musame}.
\end{Rem}

For a polynomial $h\in\C[z]$, the set of $h$-preperiodic points in $\C$ is
\[\PrePer(h,\C):=\left\lbrace z_0\in\C\colon\#O_h(z_0)<\infty\right\rbrace.\]
We prove the following proposition, whose main idea is contained in \cite[Proof of Proposition~4.5]{YZ26}.
\begin{Prop}\label{poly2musame}
	Let $f,g,c\in\C[z]$ be polynomials such that $\deg(f)\geq2$ and $\deg(g)\geq2$. Suppose that $\# Q_{f,g,c}^2=\infty$. Then $\PrePer(f,\C)=\PrePer(g,\C)$; hence, we have $\mu_f=\mu_g$ and $J(f)=J(g)$.
\end{Prop}
\begin{proof}
	If $\PrePer(f,\C)=\PrePer(g,\C)$, then $\mu_f=\mu_g$ by \cite[Theorem~1.4]{YZ21}, and we have
	\[J(f)=\supp(\mu_f)=\supp(\mu_g)=J(g).\]
	It suffices to prove $\PrePer(f,\C)=\PrePer(g,\C)$.
	
	If $f^{\circ m}=g^{\circ n}$ for some integers $m,n\geq1$, then clearly
	\[\PrePer(f,\C)=\PrePer(g,\C).\]
	Hence we may assume that $f^{\circ m}\neq g^{\circ n}$ in $\C[z]$ for all integers $m,n\geq1$. 
	
	Set $A_f=\left\lbrace m\in\Z_{\geq0}\colon f^{\circ m}=c\right\rbrace$ and $A_g=\left\lbrace n\in\Z_{\geq0}\colon g^{\circ n}=c\right\rbrace$, which are sets of cardinality at most one. By the assumption that $\# Q_{f,g,c}^2=\infty$, we can take a sequence $(\la_j)_{j=1}^\infty\subseteq Q_{f,g,c}^2$ of pairwise distinct elements, a sequence $(m_j)_{j=1}^\infty\subseteq\Z_{\geq0}\setminus A_f$, and a sequence $(n_j)_{j=1}^\infty\subseteq\Z_{\geq0}\setminus A_g$ such that for every $j\geq1$, we have
	\[f^{\circ m_j}(\la_j)=g^{\circ n_j}(\la_j)=c(\la_j).\]
	
	We claim that
	\[\lim_{j\to\infty}m_j=\lim_{j\to\infty}n_j=\infty.\]
	If the sequence $(m_j)_{j=1}^\infty$ does not tend to infinity, then there exist $m^\prime\in\Z_{\geq0}\setminus A_f$ and an infinite subset $J\subseteq\Z_{>0}$ such that $m_j=m^\prime$ for all $j\in J$. Thus, every number in the infinite set $\left\lbrace \la_j\colon j\in J\right\rbrace$ is a solution to the equation $f^{\circ m^\prime}(\la)=c(\la)$ ($\la\in\C$), which is a contradiction to $f^{\circ m^\prime}\neq c$. Hence $\lim_{j\to\infty}m_j=\infty$, and similarly $\lim_{j\to\infty}n_j=\infty$.
	
	\medskip
	
	Fix an algebraic closure $\overline{\Q}$ of $\Q$ in $\C$. Let $K$ be the subfield of $\C$ generated by all coefficients of $f$, $g$, and $c$ over $\overline{\Q}$, and $\overline{K}$ the algebraic closure of $K$ in $\C$. Then $K/\overline{\Q}$ is a finitely generated extension. In the following, let $h:\overline{\Q}\to\R_{\geq0}$ be the absolute logarithmic height over number fields if $K=\overline{\Q}$, and let $h:\overline{K}\to\R_{\geq0}$ be the standard Weil height associated to the function field $K/\overline{\Q}$ if $K\neq\overline{\Q}$. For every polynomial $c\in\C[z]$,
	\[h\circ c-\deg(c)\cdot h=O(1)\]
	is a bounded function, where we set $\deg(0)=0$. See \cite{HtBook} for more information about the height function $h$.
	
	For a polynomial $\phi\in\overline{K}[z]$ of degree $d\geq2$, let $\hat{h}_\phi:\overline{K}\to\R_{\geq0}$ be the canonical height function associated with $\phi$ as in \cite{CallSilverman}, which is given by
	\[\hat{h}_\phi(z_0)=\lim_{k\to\infty}\frac{1}{d^k}h(\phi^{\circ k}(z_0)),\quad z_0\in\overline{K}.\]
	It satisfies
	\[\hat{h}_\phi\circ\phi=d\hat{h}_\phi\quad\text{and}\quad\hat{h}_\phi-h=O(1).\]
	
	\medskip
	
	Set $d_1=\deg(f)\geq2$ and $d_2=\deg(g)\geq2$. After passing to subsequences, we may assume that $d_1^{m_j},d_2^{n_j}>\deg(c)$ for all $j\geq1$. For every $j\in\Z_{>0}$, since $f^{\circ m_j}\neq g^{\circ n_j}$ are two different polynomials with $K$-coefficients, from $f^{\circ m_j}(\la_j)=g^{\circ n_j}(\la_j)$ we get $\la_j\in\overline{K}$. Similarly, we have $\PrePer(f,\C)\cup\PrePer(g,\C)\subseteq\overline{K}$. We prove that
	\begin{equation}\label{smallpts}
		\lim_{j\to\infty}\hat{h}_f(\la_j)=\lim_{j\to\infty}\hat{h}_g(\la_j)=0.
	\end{equation}
	Let $j\geq1$. Then
	\[\hat{h}_f(\la_j)=\frac{1}{d_1^{m_j}}\hat{h}_f(f^{\circ m_j}(\la_j))=\frac{1}{d_1^{m_j}}\hat{h}_f(c(\la_j))=\frac{1}{d_1^{m_j}}(\deg(c)\hat{h}_f(\la_j)+O(1)),\]
	where the $O(1)$ is independent of $j$. Hence
	\[\hat{h}_f(\la_j)=\frac{O(1)}{d_1^{m_j}-\deg(c)}\]
	tends to $0$ as $j\to\infty$, because $\lim_{j\to\infty}m_j=\infty$. Similarly, we obtain \[\lim_{j\to\infty}\hat{h}_g(\la_j)=0.\]
	
	By \eqref{smallpts} and \cite[Proposition~2.5]{NZ25}, we conclude that
	\begin{equation}\label{htsame}
		\hat{h}_f=\hat{h}_g.
	\end{equation}
	
	\medskip
	
	\textbf{Case 1.} Suppose that $K=\overline{\Q}$. By \cite[Theorem~3.22]{Silverman2007}, we have
	\[\PrePer(f,\C)=(\hat{h}_f)^{-1}(0)=(\hat{h}_g)^{-1}(0)=\PrePer(g,\C).\]
	
	\textbf{Case 2.} Suppose that $K\neq\overline{\Q}$ and $f\in K[z]$ is non-isotrivial. Then $g\in K[z]$ is also non-isotrivial by \cite[Theorem~1.6]{Bak09} and \eqref{htsame}. Therefore, by \cite[Corollary~1.8]{Bak09}, we obtain
	\[\PrePer(f,\C)=(\hat{h}_f)^{-1}(0)=(\hat{h}_g)^{-1}(0)=\PrePer(g,\C).\]
	
	\textbf{Case 3.} Suppose that $K\neq\overline\Q$ and $f\in K[z]$ is isotrivial. After conjugacy and a finite extension of $K$, we may assume that $f\in\overline{\Q}[z]$. By \cite[Theorem~5.3]{BD11} and \eqref{smallpts}, for every place $v$ of the function field $K$, we have $\mu_{f,v}=\mu_{g,v}$, where $\mu_{f,v}$ (resp. $\mu_{g,v}$) is the canonical measure associated with $f$ on the Berkovich analytification $\P^1_{\Berk,v}$ of the projective line at $v$. Since $f\in\overline{\Q}[z]$, the canonical measure $\mu_{f,v}$ is the Dirac measure supported at the Gauss point of $\P^1_{\Berk,v}$ by \cite[Proposition~5.4]{BD11}. Then $\mu_{g,v}$ is also the Dirac measure supported at the Gauss point of $\P^1_{\Berk,v}$ for every place $v$ of $K$; hence, $g$ is also defined over $\overline{\Q}$ by \cite[Proposition~5.4]{BD11}. For every $j\geq1$, we have $\la_j\in\overline{\Q}$ because $f^{\circ m_j}(\la_j)= g^{\circ n_j}(\la_j)$ but $f^{\circ m_j}\neq g^{\circ n_j}$; then $c(\la_j)=f^{\circ m_j}(\la_j)\in\overline{\Q}$. Let
	\[C=\left\lbrace (z,c(z))\colon z\in \P^1_\C\right\rbrace\subseteq(\P^1_\C)^2\]
	be the graph of $c:\P^1_\C\to\P^1_\C$, which is an irreducible curve. Then
	\[\left\lbrace (\la_j,c(\la_j))\colon j\geq1\right\rbrace\]
	is a Zariski dense subset of $\overline{\Q}$-points in $C$; hence, $C$ and $c$ are defined over $\overline{\Q}$. We are back to Case 1.
\end{proof}

\begin{Rem}
	Let $L$ be a finitely generated subfield of $\C$, and $f,g\in L[z]$ be two polynomials of degree $\geq2$. In fact, if $\PrePer(f,\C)=\PrePer(g,\C)$, then the canonical measures associated with $f$ and $g$ are equal at every place of $L$ by \cite[Theorem~1.4]{YZ21}. See also \cite[Remark~3.7]{Tits}.
	
	In the proof of Proposition~\ref{poly2musame}, in the function field case, we choose the base field to be $\overline{\Q}$ rather than $\Q$ to make the constant field of $K$ the algebraically closed field $\overline{\Q}$.
\end{Rem}

\begin{Rem}
	It seems that the proofs of \cite[Propositions~4.2~and~4.3]{NZ25} and the proof of \cite[Proposition~4.5]{YZ26} may omit some details in the isotrivial case of function fields, which corresponds to Case 3 in the proof of Proposition~\ref{poly2musame}. Once it is known that both $f$ and $g$ are isotrivial, it is necessary to argue that $f$ and $g$ can be defined over $\overline{\Q}$ simultaneously by conjugacy over $\overline{K}$. As Case 3 of the proof of Proposition~\ref{poly2musame} showed, this problem can be solved based on the techniques in \cite{BD11}.
\end{Rem}

\medskip

Polynomials with the same measure of maximal entropy have powerful descriptions. Let $f,g\in\C[z]$ be two polynomials of degrees $d_1,d_2\geq2$, respectively. By \cite{Zdunik90}, if $J(f)=J(g)$, then:
\begin{itemize}
	\item $f$ is non-exceptional if and only if $g$ is non-exceptional;
	\item $f$ is conjugate to $z^{d_1}$ if and only if $g$ is conjugate to $z^{d_2}$;
	\item $f$ is conjugate to $\pm T_{d_1}$ if and only if $g$ is conjugate to $\pm T_{d_2}$.
\end{itemize}
It is well known that for polynomials, $\mu_f=\mu_g$ if and only if $J(f)=J(g)$. However, this result does not hold for rational maps. The following theorem gives a classification of non-exceptional polynomials with the same Julia sets, which is due to Schmidt and Steinmetz \cite{SS95} (see also \cite[Theorem~3.8]{FG22}):
\begin{Thm}\label{polysamemu}
	Let $f$ be a non-exceptional polynomial of degree $\geq2$. Then there exists a non-exceptional polynomial $h\in\C[z]$ of degree $\geq2$ such that for every $g\in\C[z]$ of degree $\geq2$, the following statements are equivalent:
	\begin{itemize}
		\item $J(g)=J(f)$ (equivalently, $\mu_f=\mu_g$);
		\item $g=\sigma\circ h^{\circ k}$ for some linear polynomial $\sigma\in\Aut(J(f))$ and some integer $k\geq1$, where
		\[\Aut(J(f)) = \left\lbrace \sigma\in\C[z] \colon \deg(\sigma)=1 \text{ and } \sigma(J(f))=J(f) \right\rbrace\]
		is the \emph{automorphism group of $J(f)$}.
	\end{itemize}
\end{Thm}

\begin{Rem}\label{AutJ}
	Let $d=\deg(h)\geq2$. After conjugacy, we may assume that $h$ is monic and centered, i.e., $h$ is of the form
	\[h(z)=z^d+a_{d-2}z^{d-2}+\cdots+a_0.\]
	Write
	\[h(z)=z^r R(z^s),\]
	where $r\geq0$ and $s\geq1$ are integers, and $R\in\C[z]$ is a polynomial such that $R(0)\neq0$ and $R$ cannot be further written as $R_0(z^l)$ for some polynomial $R_0$ and some integer $l\geq2$. Note that $R\notin\C$ because $h$ is non-exceptional. The pair $(r,s)$ is uniquely determined by $h$. In this case, we have
	\[\Aut(J(h))=\left\lbrace \zeta z\colon \zeta\in\C,\,\zeta^s=1\right\rbrace.\]
	See \cite[Proposition~3.9]{FG22} for a proof.
\end{Rem}

\begin{Def}\label{UC}
	Define
	\[U(\C)=\left\lbrace\zeta\in\C\colon \zeta^m=1\text{ for some }m\in\Z_{>0}\right\rbrace,\]
	which is the group of all roots of unity in $\C$. For every $s\in\Z_{>0}$, define
	\[U_s(\C)=\left\lbrace\zeta\in\C\colon \zeta^s=1\right\rbrace,\]
	which is the group of $s$-th roots of unity in $\C$.
\end{Def}

\medskip

To deal with the rank-two case, we need to consider a subset of $Q_{f,g,c}^2$ in Definition~\ref{Q2poly}:
\begin{Def}\label{Rpoly}
	Let $f,g,c\in\C[z]$ be polynomials such that $d_1:=\deg(f)\geq2$ and $d_2:=\deg(g)\geq2$. Define $R_{f,g,c}$ to be the set
	\[\left\lbrace \la\in\C\colon \exists (m,n)\in\Z_{\geq0}^2, f^{\circ m}(\la)=g^{\circ n}(\la)=c(\la),\,f^{\circ m}\neq c,\,g^{\circ n}\neq c,\,d_1^m\neq d_2^n\right\rbrace.\]
\end{Def}

Applying the rank-one Theorem~\ref{thmYZ}, we obtain the following result in the rank-two case:
\begin{Prop}\label{Rfinite}
	Let $f,g,c\in\C[z]$ be polynomials such that $d_1:=\deg(f)\geq2$ and $d_2:=\deg(g)\geq2$. Suppose that $\#R_{f,g,c}<\infty$. Then $S_{f,g,c}^2$ is semi-linear.
\end{Prop}
\begin{proof}
	If $\log d_1/\log d_2 \notin\Q$, then $R_{f,g,c}=Q_{f,g,c}^2$ and the conclusion follows from Proposition~\ref{Q2polyfinite}. We may assume that $\log d_1/\log d_2 \in\Q_{>0}$, and write
	\[\frac{\log d_1}{\log d_2}=\frac{p}{q},\]
	where $p$ and $q$ are coprime positive integers. If $(m,n)\in\Z_{\geq0}^2$ is such that $d_1^m=d_2^n$, then $(m,n)=(qk,pk)$ for some integer $k\geq0$.
	
	Set $F=f^{\circ q}$ and $G=g^{\circ p}$. Then
	\[\deg(F)=d_1^q=d_2^p=\deg(G)\geq2.\]
	Let $S_s$ and $S^2(\la_0)$ ($\la_0\in R_{f,g,c}$) be as in the proof of Proposition~\ref{Q2polyfinite}. Then we have a decomposition
	\begin{equation}\label{S2decompR}
		S_{f,g,c}^2=\sA_R\cup S_s\cup\bigcup_{\la_0\in R_{f,g,c}}S^2(\la_0),
	\end{equation}
	where
	\[\sA_R=\left\lbrace(qk,pk)\colon k\in\Z_{\geq0}\text{ such that }\exists\la\in\C,\,f^{\circ qk}(\la)=g^{\circ pk}(\la)=c(\la)\right\rbrace.\]
	As in the proof of Proposition~\ref{Q2polyfinite}, the set $S_s$ is semi-linear by Lemma~\ref{lemSs}, and $S^2(\la_0)=S_{f,c}(\la_0)\times S_{g,c}(\la_0)$ is semi-linear for all $\la_0\in R_{f,g,c}$ by \eqref{Shla}.
	
	By Theorem~\ref{thmYZ} (in fact, a special case \cite[Proposition~4.5]{YZ26} is sufficient), the rank-one recurrence set
	\[S_{F,G,c}=\left\lbrace k\in\Z_{\geq0}\colon \exists\la\in\C,\,F^{\circ k}(\la)=G^{\circ k}(\la)=c(\la)\right\rbrace\]
	is semi-linear. Hence $\sA_R=\left\lbrace (qk,pk)\colon k\in S_{F,G,c}\right\rbrace $ is semi-linear by Theorem~\ref{thmGS}. Therefore, the decomposition \eqref{S2decompR} implies that $S_{f,g,c}^2$ is also semi-linear.
\end{proof}

Under the condition $\# R_{f,g,c}=\infty$, the following proposition studies the structure of $c$, which is a slight modification of Yang--Zhong's result \cite[Proposition~4.6]{YZ26}:
\begin{Prop}\label{poly3cases}
	Let $f,g,c\in\C[z]$ be polynomials such that $\deg(f)\geq2$ and $\deg(g)\geq2$. Suppose that $\# R_{f,g,c}=\infty$. By Proposition~\ref{poly2musame}, we have
	\[P:=\PrePer(f,\C)=\PrePer(g,\C),\quad\mu:=\mu_f=\mu_g,\quad\text{and}\quad J:=J(f)=J(g).\]
	\begin{enumerate}
		\item If $c\in\C$, then $c\in P$.
		\item If $\deg(c)=1$, then $c(P)=P$ and $c(J)=J$.
		\item If $\deg(c)\geq2$, then $\PrePer(c,\C)=P$, $\mu_c=\mu$, and $J(c)=J$.
	\end{enumerate}
\end{Prop}
\begin{proof}
	Set
	\[A_f=\left\lbrace m\in\Z_{\geq0}\colon f^{\circ m}=c\right\rbrace\quad\text{and}\quad A_g=\left\lbrace n\in\Z_{\geq0}\colon g^{\circ n}=c\right\rbrace.\]
	Then $\# A_f\leq1$ and $\# A_g\leq1$. Since $\# R_{f,g,c}=\infty$, we may take a sequence $(\la_j)_{j=1}^\infty\subseteq R_{f,g,c}$ of pairwise distinct elements, a sequence $(m_j)_{j=1}^\infty\subseteq\Z_{\geq0}\setminus A_f$, and a sequence $(n_j)_{j=1}^\infty\subseteq\Z_{\geq0}\setminus A_g$ such that for every $j\geq1$, we have
	\[d_1^{m_j}>d_2^{n_j}\quad\text{and}\quad f^{\circ m_j}(\la_j)=g^{\circ n_j}(\la_j)=c(\la_j).\]
	Let
	\[C=\left\lbrace (z,c(z))\colon z\in\P^1_\C\right\rbrace\subseteq(\P^1_\C)^2\]
	be the graph of $c:\P^1_\C\to\P^1_\C$, which is an irreducible curve.
	
	Since $\mu_f=\mu_g$, by \cite{Zdunik90}, there are three possibilities:
	\begin{itemize}
		\item both $f$ and $g$ are non-exceptional;
		\item $f$ is conjugate to $z^{d_1}$ and $g$ is conjugate to $z^{d_2}$;
		\item $f$ is conjugate to $\pm T_{d_1}$ and $g$ is conjugate to $\pm T_{d_2}$.
	\end{itemize}
	
	\medskip

	\textbf{Case 1.} Suppose that both $f$ and $g$ are non-exceptional.
	
	Since $f$ and $g$ have the same measure of maximal entropy $\mu$, by Theorem~\ref{polysamemu} and Remark~\ref{AutJ}, after affine conjugacy, we may assume that
	\begin{equation}\label{fgform1}
		f(z)=\zeta_1 h^{\circ k_1}(z)\quad\text{and}\quad g(z)=\zeta_2 h^{\circ k_2}(z),
	\end{equation}
	where $k_1,k_2\geq1$ are integers, $h\in\C[z]$ is a monic centered non-exceptional polynomial of degree $\geq2$ with $J(h)=J$ of the form
	\begin{equation}\label{hform1}
		h(z)=z^r R(z^s)\quad(r\geq1,\,s\geq1,\,\text{and }R\in\C[z]\setminus\C)
	\end{equation}
	as in Remark~\ref{AutJ}, and $\zeta_1,\zeta_2\in U_s(\C)$ are $s$-th roots of unity. Set $d:=\deg(h)\geq2$. Then $d_1=d^{k_1}$ and $d_2=d^{k_2}$.
	
	By \eqref{fgform1} and \eqref{hform1}, it is easy to see that $\PrePer(h,\C)=P$, i.e., $f$, $g$, and $h$ have the same set of preperiodic points. For $l\in\Z_{>0}$, since $d\equiv r\Mod{s}$, we compute that
	\begin{equation}\label{fgiter1}
		f^{\circ l}(z)=\zeta_1^{(d^{k_1 l}-1)/(d^{k_1}-1)} h^{\circ k_1 l}(z)\quad\text{and}\quad g^{\circ l}(z)=\zeta_2^{(d^{k_2 l}-1)/(d^{k_2}-1)}h^{\circ k_2 l}(z).
	\end{equation}
	Let $j\geq1$. Note that $k_1 m_j>k_2 n_j$ because $d_1^{m_j}>d_2^{n_j}$. Set $\mu_j=h^{\circ k_2 n_j}(\la_j)$. From \eqref{fgiter1} and $f^{\circ m_j}(\la_j)=g^{\circ n_j}(\la_j)$, we obtain
	\begin{equation}\label{etapreper}
		\mu_j=h^{\circ k_2 n_j}(\la_j)=\eta_j h^{\circ k_1 m_j}(\la_j)=\eta_j h^{\circ(k_1 m_j-k_2 n_j)}(\mu_j)
	\end{equation}
	for some $s$-th root of unity $\eta_j\in U_s(\C)$. Since $k_1 m_j-k_2 n_j>0$, by \eqref{etapreper} and \eqref{hform1} it is easy to see that $\mu_j\in\PrePer(h,\C)$; hence, $\la_j\in\PrePer(h,\C)=P$. Therefore,
	\[c(\la_j)=f^{\circ m_j}(\la_j)\in\PrePer(f,\C)=P.\]
	
	Then the graph $C$ of $c$ contains infinitely many $(h,h)$-preperiodic points
	\[\left\lbrace (\la_j,c(\la_j))\colon j\geq1 \right\rbrace,\]
	where $(h,h)$ is the split endomorphism
	\[(h,h):(\P^1_\C)^2\to(\P^1_\C)^2,\,(x,y)\mapsto(h(x),h(y)).\]
	If $c\in\C$, then $c\in P$. Assume now $\deg(c)\geq1$. By \cite[Theorem~1.5]{GNY19}, there exist integers $N_1,N_2\geq1$, $l\geq k\geq0$, a polynomial $p\in\C[z]$ of degree $\geq2$, and a polynomial $\sigma\in\C[z]$ of degree $1$ such that
	\begin{equation}\label{prepercurve}
		p^{\circ l}=\sigma\circ p^{\circ k}\circ c
	\end{equation}
	and
	\begin{equation}\label{psigcomm}
		p\circ h^{\circ N_1}=h^{\circ N_1}\circ p\quad\text{and}\quad\sigma\circ h^{\circ N_2}=h^{\circ N_2}\circ\sigma.
	\end{equation}
	
	From \eqref{psigcomm}, it is easy to see that
	\[\PrePer(p,\C)=P\quad\text{and}\quad\sigma(J)=J.\]
	In fact, for $z_0\in P=\PrePer(h^{\circ N_1},\C)$, assume that $h^{\circ mN_1}(z_0)=h^{\circ nN_1}(z_0)$ for some integers $m>n\geq0$. For every $k\geq0$, we have
	\[h^{\circ mN_1}\circ p^{\circ k}(z_0)=p^{\circ k}\circ h^{\circ mN_1}(z_0)=p^{\circ k}\circ h^{\circ nN_1}(z_0)=h^{\circ nN_1}\circ p^{\circ k}(z_0);\]
	hence, the forward $p$-orbit $O_p(z_0)$ is contained in the finite set
	\[\left\lbrace z\in\C\colon h^{\circ mN_1}(z)=h^{\circ nN_1}(z)\right\rbrace\]
	and $z_0\in\PrePer(p,\C)$. We have shown that $P\subseteq\PrePer(p,\C)$, and the converse inclusion can be proved similarly. The equality $\sigma(J)=J$ holds because $\sigma$ commutes with $h^{\circ N_2}$ and
	\[J(h^{\circ N_2})=J(h)=J.\]
	
	By \cite[Theorem~1.4]{YZ21}, we have $\mu_{p}=\mu$ and $J(p)=J$. Then Theorem~\ref{polysamemu} and Remark~\ref{AutJ} imply that
	\begin{equation}\label{psigform}
		p(z)=\mu_1 h^{\circ M}(z)\quad\text{and}\quad\sigma(z)=\mu_2 z,
	\end{equation}
	where $\mu_1,\mu_2\in U_s(\C)$, and $M=\log_d\deg(p)$ is a positive integer. Substituting \eqref{psigform} into \eqref{prepercurve}, we obtain
	\begin{equation}\label{hc}
		h^{\circ kM}\circ c(z)=\mu h^{\circ lM}(z)
	\end{equation}
	for $\mu\in U_s(\C)$.
	
	Let $z_0\in P$. Set
	\[\Omega(z_0)=\left\lbrace \zeta y\colon y\in O_h(z_0)\text{ and }\zeta\in U_s(\C)\right\rbrace\quad\text{and}\quad\La(z_0)=h^{-kM}(\Omega(z_0)),\]
	which are finite subsets of $P$ by \eqref{hform1}. Clearly, we have $z_0\in\La(z_0)$. Evaluating \eqref{hc} at $z=z_0$, we get $c(z_0)\in\La(z_0)\subseteq P$. Suppose that we have shown $c^{\circ n}(z_0)\in\La(z_0)$ where $n\in\Z_{>0}$. Evaluating \eqref{hc} at $z=c^{\circ n}(z_0)$, by \eqref{hform1} and the inductive hypothesis we have
	\[h^{\circ kM}(c^{\circ (n+1)}(z_0))=\mu h^{\circ (l-k)M}(h^{\circ kM}(c^{\circ n}(z_0)))\subseteq \mu h^{\circ (l-k)M}(\Omega(z_0))\subseteq\Omega(z_0);\]
	hence, $c^{\circ (n+1)}(z_0)\in\La(z_0)$. By induction on $n$, we obtain $O_c(z_0)\subseteq\La(z_0)$. Hence $z_0$ is $c$-preperiodic. We have proven that $P\subseteq\PrePer(c,\C)$ and $c(P)\subseteq P$.
	
	Suppose that $\deg(c)\geq2$. Since $P\subseteq\PrePer(c,\C)$, we obtain $\mu_c=\mu$ by \cite[Theorem~1.4]{YZ21}; hence, $J(c)=J$.
	
	Suppose that $\deg(c)=1$. By $c(P)\subseteq P$, the intersection $\PrePer(c\circ h\circ c^{-1},\C)\cap P$ contains the infinite set $c(P)$. Then
	\[\mu_h=\mu_{c\circ h\circ c^{-1}}=c_*\mu_h\]
	by \cite[Theorem~1.4]{YZ21}. Taking supports shows $c(J)=J$. By Remark~\ref{AutJ}, $c(z)=\eta z$ for some $\eta\in U_s(\C)$. From the expression \eqref{hform1}, we deduce that $c(P)=P$.
	
	\medskip
	
	\textbf{Case 2.} Suppose that $f$ is conjugate to $z^{d_1}$ and $g$ is conjugate to $z^{d_2}$.
	
	After affine conjugacy, we may assume that
	\[f(z)=z^{d_1}\quad\text{and}\quad g(z)=\sigma\circ z^{d_2}\circ\sigma^{-1}(z)\]
	for some $\sigma\in\C[z]$ of degree $1$. Since all power polynomials $z^k$ ($k\geq2$) have Julia set $\partial\D=\{z\in\C\colon|z|=1\}$, we have
	\[\partial\D=J(f)=J(g)=\sigma(J(z^{d_2}))=\sigma(\partial\D),\]
	which implies that $\sigma(z)=az$ for some $a\in\partial\D$. Then $g(z)=a^{1-d_2}z^{d_2}$. From $\PrePer(f,\C)=\PrePer(g,\C)$, we see that $\zeta:=a^{1-d_2}\in U(\C)$ is a root of unity.
	
	Note that $J=\partial\D$ and $P=\{0\}\cup U(\C)$. After removing finitely many terms, we may assume that $\la_j\neq0$ for all $j\geq1$. Let $j\geq1$. Since
	\[\la_j^{d_1^{m_j}}=f^{\circ m_j}(\la_j)=g^{\circ n_j}(\la_j)=\zeta^{(d_2^{n_j}-1)/(d_2-1)}\la_j^{d_2^{n_j}}\]
	and $d_1^{m_j}>d_2^{n_j}$, we have $\la_j\in U(\C)$. Hence $c(\la_j)=f^{\circ m_j}(\la_j)=\la_j^{d_1^{m_j}}$ is also a root of unity. In particular, $c$ is not the constant $0$. Then the graph $C$ of $c$ contains infinitely many points
	\[\left\lbrace (\la_j,c(\la_j))\colon j\geq1\right\rbrace\]
	with coordinates in $U(\C)$. If $c\in\C$, then $c\in U(\C)\subseteq P$. Assume now $\deg(c)\geq1$. Applying the torsion point theorem (i.e., Manin--Mumford conjecture for $\G_{m,\C}^N$) \cite[Theorem~1.1]{UnlikelyBook} to the irreducible curve $C$, we conclude that $c$ is of the form
	\[c(z)=\eta z^{d_3},\]
	for some $\eta\in U(\C)$ and $d_3=\deg(c)\geq1$. If $d_3=1$, then $c(P)=P$ and $c(J)=J$. If $d_3\geq2$, then $\PrePer(c,\C)=P$, $\mu_c=\mu$, and $J(c)=J$.
	
	\medskip
	
	\textbf{Case 3.} Suppose that $f$ is conjugate to $\pm T_{d_1}$ and $g$ is conjugate to $\pm T_{d_2}$.
	
	After affine conjugacy, we may assume that
	\[f(z)=\varepsilon_1 T_{d_1}(z)\quad\text{and}\quad g(z)=\sigma\circ (\varepsilon_2^\prime T_{d_2})\circ\sigma^{-1}(z)\]
	for some $\varepsilon_1,\varepsilon_2^\prime\in\{\pm1\}$ and some $\sigma\in\C[z]$ of degree $1$. Since
	\[J(\pm T_d)=\left[-2,2\right] \subseteq\R\]
	for every $d\geq2$, we obtain
	\[\left[-2,2\right]=J(f)=J(g)=\sigma(J(\varepsilon_2^\prime T_{d_2}))=\sigma(\left[-2,2\right]),\]
	which implies that $\sigma(z)=az$ for some $a\in\{\pm 1\}$. Then $g(z)=\varepsilon_2 T_{d_2}(z)$, where $\varepsilon_2=a^{1-d_2}\varepsilon_2^\prime\in\{\pm1\}$.
	
	Let $\pi:\P^1_\C\to\P^1_\C$ be the endomorphism given by $\pi(z)=z+z^{-1}$. Then for all integers $d\geq2$,
	\begin{equation}\label{piTd}
		\pi\circ z^d=T_d\circ\pi(z)\quad\text{and}\quad\pi\circ(-z^d)=-T_d\circ\pi(z). 
	\end{equation}
	It is straightforward to verify that
	\[\PrePer(\pm T_d,\C)=\{\zeta+\zeta^{-1}\colon\zeta\in U(\C)\}=P.\]
	For each $j\geq1$, fix a preimage $\mu_j\in\pi^{-1}(\la_j)\subseteq\C^\times$. After passing to subsequences, we may assume that $(\mu_j)_j$ is a sequence of pairwise distinct non-zero complex numbers. Let $j\geq1$. From $f^{\circ m_j}(\la_j)=g^{\circ n_j}(\la_j)$ and \eqref{piTd}, we get
	\begin{equation}\label{Tdj1}
		\pi\left(\mu_j^{d_1^{m_j}}\right)=\gamma_j\pi\left(\mu_j^{d_2^{n_j}}\right)=\pi\left(\gamma_j\mu_j^{d_2^{n_j}}\right),
	\end{equation}
	where $\gamma_j=\varepsilon_2^{(d_2^{n_j}-1)/(d_2-1)}\varepsilon_1^{-(d_1^{m_j}-1)/(d_1-1)}\in\{\pm1\}$. By \eqref{Tdj1}, we conclude that
	\[\mu_j^{d_1^{m_j}}=\gamma_j\mu_j^{d_2^{n_j}}\quad\text{or}\quad\mu_j^{d_1^{m_j}}=\gamma_j^{-1}\mu_j^{-d_2^{n_j}};\]
	hence, $0\neq\mu_j$ is a root of unity because $d_1^{m_j}\pm d_2^{n_j}>0$.
	
	Let $C_0:=(\pi,\pi)^{-1}(C)\subseteq(\P^1_\C)^2$ be the preimage of the graph $C$ of $c$ under the split endomorphism $(\pi,\pi):(\P^1_\C)^2\to(\P^1_\C)^2$. We have $\dim C_0=1$. Set
	\[p_j=\left(\mu_j,\varepsilon_{1j}\mu_j^{d_1^{m_j}}\right) \in U(\C)^2,\]
	where $\varepsilon_{1j}=\varepsilon_1^{(d_1^{m_j}-1)/(d_1-1)}\in\{\pm1\}$. Then
	\[(\pi,\pi)(p_j)=\left(\la_j,\varepsilon_{1j} T_{d_1^{m_j}}(\pi(\mu_j))\right)=\left(\la_j,f^{\circ m_j}(\la_j)\right)=\left(\la_j,c(\la_j)\right)\in C;\]
	hence, $p_j\in C_0$. We have shown that $C_0$ contains infinitely many complex points $\{p_j\colon j\geq1\}$ with coordinates in $U(\C)$. Applying the torsion point theorem \cite[Theorem~1.1]{UnlikelyBook} to $C_0$, we conclude that there exist coprime integers $a,b\in\Z\setminus\{0\}$ and $\eta\in U(\C)$ such that the (set-theoretic) image of the morphism
	\[\varphi:\P^1_\C\to(\P^1_\C)^2,\,t\mapsto(t^a,\eta t^b)\]
	is contained in $C_0$. Then $\mathrm{Im}((\pi,\pi)\circ\varphi)\subseteq C$ as sets, which implies that
	\begin{equation}\label{cab}
		c(t^a+t^{-a})=\eta t^b+\eta^{-1} t^{-b}\quad\text{in}\quad\C(t).
	\end{equation}
	Taking degrees for \eqref{cab}, we get $\deg(c)\cdot|a|=|b|$. Then $|a|=1$ because $\gcd(a,b)=1$. So \eqref{cab} becomes
	\[c(t+t^{-1})=\eta t^b+\eta^{-1} t^{-b}=\eta T_{|b|}(t+t^{-1})+(\eta^{-1}-\eta)t^{-b},\]
	where we define $T_1(z)=z$. Hence $(\eta^{-1}-\eta)t^{-b}$ is a polynomial in $t+t^{-1}$, which is possible only if $\eta^{-1}-\eta=0$, i.e., $\eta=\pm1$. Thus,
	\[c(z)=\pm T_{|b|}(z).\]
	If $\deg(c)=|b|=1$, then $c(z)=\pm z$, and it is easy to see that $c(P)=P$ and $c(J)=J$. If $\deg(c)=|b|\geq2$, then $\PrePer(c,\C)=P$, $\mu_c=\mu$, and $J(c)=J$.
\end{proof}

\subsection{The case of rational maps}\label{secsamemurat}
\begin{Def}
	Let $f,g,c\in\C(z)$ be rational maps such that $\deg(f)\geq2$ and $\deg(g)\geq2$. Define
	\[\hat{Q}_{f,g,c}:=\left\lbrace \la\in\P^1(\C)\colon \exists n\in\Z_{\geq0}, f^{\circ n}(\la)=g^{\circ n}(\la)=c(\la)\right\rbrace.\]
\end{Def}

\begin{Prop}\label{Qratfinite}
	Let $f,g,c\in\C(z)$ be rational maps such that $\deg(f)\geq2$ and $\deg(g)\geq2$. If $\# \hat{Q}_{f,g,c}<\infty$, then $\hat{S}_{f,g,c}$ is semi-linear.
\end{Prop}
\begin{proof}
	The proof is similar to the proof of Proposition~\ref{Q2polyfinite}; see also \cite[Proof of Theorem~4.8]{YZ26}.
\end{proof}

For a rational map $h\in\C(z)$, the set of $h$-preperiodic points in $\P^1(\C)$ is denoted by $\PrePer(h,\P^1(\C))$. The following proposition is an analogue of Proposition~\ref{poly2musame} in the rank-one case of rational maps:
\begin{Prop}\label{ratmusame}
	Let $f,g,c\in\C(z)$ be rational maps such that $\deg(f)\geq2$ and $\deg(g)\geq2$. Suppose that $\# \hat{Q}_{f,g,c}=\infty$. Then $\PrePer(f,\P^1(\C))=\PrePer(g,\P^1(\C))$; hence, we have $\mu_f=\mu_g$ and $J(f)=J(g)$.
\end{Prop}
\begin{proof}
	The proof is similar to the proof of Proposition~\ref{poly2musame}; see also \cite[Proof of Theorem~4.5]{YZ26}.
\end{proof}

\begin{Rem}\label{Q2ratfinite}
	The argument in Propositions~\ref{Q2polyfinite}~and~\ref{poly2musame} for the rank-two case of polynomials also applies to rational maps. Precisely, let $f,g,c\in\C(z)$ be rational maps with $\deg(f)\geq2$ and $\deg(g)\geq2$, and define
	\[\hat{Q}_{f,g,c}^2:=\left\lbrace\la\in\P^1(\C)\colon\exists(m,n)\in\Z_{\geq0}^2,f^{\circ m}(\la)=g^{\circ n}(\la)=c(\la),f^{\circ m}\neq c,g^{\circ n}\neq c\right\rbrace.\]
	Then we can prove:
	\begin{itemize}
		\item If $\# \hat{Q}_{f,g,c}^2<\infty$, then $\hat{S}_{f,g,c}^2$ is semi-linear.
		\item If $\# \hat{Q}_{f,g,c}^2=\infty$, then $\PrePer(f,\P^1(\C))=\PrePer(g,\P^1(\C))$, and hence $\mu_f=\mu_g$ and $J(f)=J(g)$.
	\end{itemize}
\end{Rem}

The following theorem of Levin and Przytycki \cite{LP97} describes non-exceptional rational maps with the same measure of maximal entropy (see \cite[Theorem~1.6]{Ye} and \cite{Paksamemu}):
\begin{Thm}[Levin--Przytycki]\label{ratsamemu}
	Let $f,g\in\C(z)$ be two non-exceptional rational maps of degree $\geq2$. Then the following statements are equivalent:
	\begin{enumerate}
		\item $\mu_f=\mu_g$;
		\item \label{ratsamemu2} there exist integers $l,k\geq1$ such that $f^{\circ 2l}=f^{\circ l}\circ g^{\circ k}$ and $g^{\circ 2k}=g^{\circ k}\circ f^{\circ l}$.
	\end{enumerate}
\end{Thm}

\begin{Rem}
	Note that the integers $l,k$ in (\ref{ratsamemu2}) satisfy $\deg(f)^l=\deg(g)^k$. Hence if two non-exceptional rational maps $f,g\in\C(z)$ of degree $\geq2$ have the same measure of maximal entropy, then $\log_{\deg(f)}\deg(g)\in\Q_{>0}$.
	
	Unlike the polynomial case (see Theorem~\ref{polysamemu}), Ye showed \cite[Theorem~1.1]{Ye} that there exist non-exceptional rational functions $f$ and $g$ of degree $\geq2$, satisfying:
	\begin{itemize}
		\item $\mu_f=\mu_g$;
		\item $f^{\circ l}\neq\sigma\circ g^{\circ k}$ for all $l,k\in\Z_{>0}$ and all $\sigma\in\C(z)$ of degree $1$.
	\end{itemize}
\end{Rem}

\begin{Cor}\label{rat1to2}
Theorem~\ref{thmratg}~(\ref{rat1g}) implies Theorem~\ref{thmratg}~(\ref{rat2g}).
\end{Cor}
\begin{proof}
	Let $f,g,c\in\C(z)$ be rational maps such that $\deg(f)=\deg(g)\geq2$ and $f$ is non-exceptional. We want to prove that the rank-one recurrence set $\hat{S}_{f,g,c}$ is semi-linear.
	
	By Proposition~\ref{Qratfinite}, we may assume that $\# \hat{Q}_{f,g,c}=\infty$. Then Proposition~\ref{ratmusame} yields $\mu_f=\mu_g$. Since $f$ is non-exceptional, it follows from \cite{Zdunik90} that $g$ is also non-exceptional. Applying Theorem~\ref{ratsamemu}, we obtain two integers $l,k\geq1$ such that
	\[f^{\circ 2l}=f^{\circ l}\circ g^{\circ k}\quad\text{and}\quad g^{\circ 2k}=g^{\circ k}\circ f^{\circ l}.\]
	Note that $k=l$ because $\deg(f)=\deg(g)$. Hence $\hat{S}_{f,g,c}$ is semi-linear by Theorem~\ref{thmratg}~(\ref{rat1g}).
\end{proof}

\section{Proofs of main theorems}\label{secpf}
\subsection{Rank-two recurrence for exceptional polynomials}\label{secexcep}
We deal with the exceptional case of Theorem~\ref{thmrk2poly} in this subsection.

The following lemma is a generalized version of \cite[Lemma~4.1]{YZ26}. We state it without proof because the proof is purely elementary and identical to that of \cite[Lemma~4.1]{YZ26}.
\begin{Lem}\label{lem4.1}
	Let $k\in\Z_{>0}$ and $\xi\in U(\C)$ be a primitive $k$-th root of unity. Let $a,b,d_1,d_2,d_3,d_4$ be integers such that $|d_1|,|d_2|\geq2$. Then for every pair $(m,n)\in\Z_{\geq0}^2$ with $|d_1|^m>|d_3|$ and $|d_2|^n>|d_4|$, the following two statements are equivalent:
	\begin{enumerate}
		\item $\exists\la\in\C^\times$, $\la^{d_1^m-d_3}=\xi^a$ and $\la^{d_2^n-d_4}=\xi^b$;
		\item $k\cdot\gcd\left(|d_1^m-d_3|,|d_2^n-d_4|\right)\mid\left(b(d_1^m-d_3)-a(d_2^n-d_4)\right)$.
	\end{enumerate}
\end{Lem}

\smallskip

We first deal with power maps.
\begin{Prop}\label{powerconst}
	Let $d_1, d_2\in\Z\setminus\{0,\pm1\}$ be integers, and let $\zeta,c\in U(\C)$. Define $f(z)=z^{d_1}$ and $g(z)=\zeta z^{d_2}$, which are rational maps of degree $\geq2$ (which need not be polynomials). Note that
	\[\hat{S}_{f,g,c}^2=\left\lbrace(m,n)\in\Z_{\geq0}^2\colon\exists\la\in\C^\times,\,f^{\circ m}(\la)=g^{\circ n}(\la)=c\right\rbrace\]
	because $c\in\C^\times$. Then $\hat{S}_{f,g,c}^2$ is semi-linear.
\end{Prop}
\begin{proof}
	For $m\in\Z_{\geq0}$, we compute $f^{\circ m}(z) = z^{d_1^m}$, where $f^{\circ 0}(z)=z^{d_1^0}=z$ by convention. For the power map $g$, we have $g^{\circ 0}(z)=z$, and
	\[g^{\circ n}(z)=\zeta^{1+d_2+\cdots+d_2^{n-1}}z^{d_2^n}=\zeta^{(d_2^n-1)/(d_2-1)}z^{d_2^n}\]
	for $n\in\Z_{>0}$. It is easy to see that
	\[\hat{S}_{f,g,c}^2\cap\left(\Z_{\geq0}\times\{0\}\right)=\left\lbrace (m,0)\colon m\in\Z_{\geq0}\text{ such that }c^{d_1^m-1}=1\right\rbrace\]
	is semi-linear (see the proof of Lemma~\ref{lemSs}). Thus, it remains to show that
	\[S^\prime:=\hat{S}_{f,g,c}^2\cap\left(\Z_{\geq0}\times\Z_{>0}\right)\]
	is semi-linear.
	
	\medskip
	
	Let $k\in\Z_{>0}$ be the least common multiple of the orders of the roots of unity $\zeta$ and $c$. Set $\xi=\exp(2\pi i/k)$. Let $a,e\in\{0,1,\dots,k-1\}$ be determined by $\xi^a=c$ and $\xi^e=\zeta$. Note that $(k,a,e)$ are determined by $(d_1,d_2,\zeta,c)$.
	
	Set
	\[Q_0=d_2-1\quad\text{and}\quad Q=\left|Q_0\right|.\]
	Then $Q$ is a positive integer. For $(m,n)\in\Z_{\geq0}\times\Z_{>0}$, by applying Lemma~\ref{lem4.1} with $(d_3,d_4,b)=\left(0,0,a-e(d_2^n-1)/(d_2-1)\right)$ and $(d_1,d_2,k,\xi,a)$ as above, we obtain:
	\begin{align*}
		&(m,n)\in S^\prime\text{, i.e., }\exists\la\in\C^\times,\,\la^{d_1^m}=\xi^{e(d_2^n-1)/(d_2-1)}\la^{d_2^n}=\xi^a\\
		\iff& kQ\cdot\gcd\left(\left|d_1^m\right|,\left|d_2^n\right|\right)\mid\left(aQ_0(d_1^m-d_2^n)-ed_1^m(d_2^n-1)\right).
	\end{align*}
	Define a function $h:\Z_{\geq0}\times\Z_{>0}\to\Z$ by
	\[h(m,n)=\frac{aQ_0(d_1^m-d_2^n)-ed_1^m(d_2^n-1)}{\gcd\left(\left|d_1^m\right|,\left|d_2^n\right|\right)}.\]
	Then
	\begin{equation}\label{Sprime}
		S^\prime=\left\lbrace (m,n)\in\Z_{\geq0}\times\Z_{>0}\colon h(m,n)\equiv 0\Mod{kQ}\right\rbrace. 
	\end{equation}
	
	\medskip
	
	Define
	\[P:=\left\lbrace p\in\Z_{\geq2}\colon p\text{ is prime and }p\mid d_1d_2\right\rbrace,\]
	which is a non-empty finite set by the fundamental theorem of arithmetic. Let $\varepsilon_1,\varepsilon_2\in\{\pm1\}$ be the signs of $d_1$ and $d_2$, respectively. Write
	\[d_1=\varepsilon_1\prod_{p\in P}p^{\tau_p}\quad\text{and}\quad d_2=\varepsilon_2\prod_{p\in P}p^{\mu_p},\]
	where $\tau_p,\mu_p\in\Z_{\geq0}$ with $\tau_p+\mu_p>0$ for $p\in P$. Then for $(m,n)\in\Z_{\geq0}\times\Z_{>0}$, we have
	\begin{equation}\label{gcdmn}
		\gcd\left(\left|d_1^m\right|,\left|d_2^n\right|\right)=\prod_{p\in P}p^{\min\{m\tau_p,n\mu_p\}}.
	\end{equation}
	
	Note that we can partition $\Z_{\geq0}\times\Z_{>0}$ by
	\begin{equation}\label{part}
		\Z_{\geq0}\times\Z_{>0}=\bigsqcup_{j\in J}C_j,
	\end{equation}
	where $J$ is a non-empty finite index set, and for each $j\in J$, there is a subset $P_j\subseteq P$ such that
	\begin{equation}\label{cone}
		C_j=\left\lbrace (m,n)\in\Z_{\geq0}\times\Z_{>0}\colon \forall p\in P,\,m\tau_p\geq n\mu_p\text{ if and only if }p\in P_j\right\rbrace
	\end{equation}
	and $C_j\neq\emptyset$. Then the map $J\to 2^P,\,j\mapsto P_j$ is injective because \eqref{part} is a disjoint union. By \eqref{part}, it suffices to prove that
	\[S^\prime_j:=S^\prime\cap C_j\]
	is semi-linear for all $j\in J$.
	
	\smallskip
	
	Fix an arbitrary $j\in J$. The goal is to show the semi-linearity of $S^\prime_j$.
	
	Note that $C_j$ is the set of integral points contained in a rational polyhedral cone $\subseteq\R_{\geq0}\times\R_{\geq1}$, and that $C_j$ is definable in the Presburger arithmetic $\Pr$. For $p\in P$, define a function $e^j_p:C_j\to\Z_{\geq0}$ by
	\[e^j_p(m,n)=\tau_p m-\mu_p n\,\text{ if }\,p\in P_j,\quad\text{and}\quad e^j_p(m,n)=\mu_p n-\tau_p m\,\text{ if }\,p\in P\setminus P_j.\]
	By \eqref{cone}, for all $p\in P$, the function $e^j_p:C_j\to\Z_{\geq0}$ is definable in $\Pr$.
	
	\smallskip
	
	Define
	\[T=\left\lbrace (0,0),(0,1),(1,0),(1,1)\right\rbrace.\]
	For $t=(t_1,t_2)\in T$, define
	\[C_{j,t}=\left\lbrace(m,n)\in C_j\colon 2\mid(m+t_1)\text{ and }2\mid(n+t_2)\right\rbrace,\]
	which is also definable in $\Pr$. We have a decomposition
	\[C_j=\bigsqcup_{t\in T}C_{j,t}.\]
	
	Take an arbitrary $t=(t_1,t_2)\in T$. Let $(m,n)\in C_{j,t}$. By \eqref{gcdmn}, we obtain
	\[\gcd\left(\left|d_1^m\right|,\left|d_2^n\right|\right)=\left(\prod_{p\in P_j}p^{n\mu_p}\right)\left(\prod_{p\in P\setminus P_j}p^{m\tau_p}\right).\]
	Then
	\[h(m,n)=h_1(m,n)-h_2(m,n)-h_3(m,n),\]
	where
	\[h_1(m,n)=(aQ_0+e)\varepsilon_1^{t_1}\prod_{p\in P_j}p^{e^j_p(m,n)},\quad h_2(m,n)=aQ_0\varepsilon_2^{t_2}\prod_{p\in P\setminus P_j}p^{e^j_p(m,n)},\]
	and
	\[h_3(m,n)=e\varepsilon_1^{t_1}\varepsilon_2^{t_2}\left(\prod_{p\in P_j}p^{\tau_p m}\right)\left(\prod_{p\in P\setminus P_j}p^{\mu_p n}\right).\]
	Observe that for all $p\in P$, the functions
	\[(m,n)\mapsto\tau_p m\quad\text{and}\quad (m,n)\mapsto\mu_p n\]
	from $\Z_{\geq0}^2$ to $\Z_{\geq0}$ are also definable in $\Pr$.
	
	For all $u,v\in\Z$ and all $l\in\Z_{>0}$, since $(u^n\Mod{l})_{n=0}^\infty$ is an eventually periodic sequence, the set
	\begin{equation}\label{epcong}
		\left\lbrace n\in\Z_{\geq0}\colon u^n\equiv v\Mod{l}\right\rbrace\quad\text{is semi-linear.}
	\end{equation}
	Here we set $0^0=1$. Consider the congruence equation
	\begin{equation}\label{cong}
		(aQ_0+e)\varepsilon_1^{t_1}\prod_{p\in P_j}x_p-aQ_0\varepsilon_2^{t_2}\prod_{p\in P\setminus P_j}y_p-e\varepsilon_1^{t_1}\varepsilon_2^{t_2}\prod_{p\in P}z_p\equiv 0\Mod{kQ}
	\end{equation}
	in $2\#P$ variables $\left((x_p)_{p\in P_j},(y_p)_{p\in P\setminus P_j},(z_p)_{p\in P}\right)$. Let
	\[W_t\subseteq\{0,1,\dots,kQ-1\}^{2\# P}\]
	be the set of all solutions of \eqref{cong} over the least residue system $\{0,1,\dots,kQ-1\}$ of $kQ$. Let
	\[w=\left((\alpha_p)_{p\in P_j},(\beta_p)_{p\in P\setminus P_j},(\gamma_p)_{p\in P}\right)\in W_t\]
	be a solution of \eqref{cong}. Define $S_{j,t,w}$ to be the set of all pairs $(m,n)\in C_{j,t}$ satisfying
	\[p^{e^j_p(m,n)}\equiv\alpha_p\Mod{kQ},\quad p^{\tau_p m}\equiv\gamma_p\Mod{kQ},\quad\text{for }p\in P_j,\]
	and
	\[p^{e^j_p(m,n)}\equiv \beta_p\Mod{kQ},\quad p^{\mu_p n}\equiv \gamma_p\Mod{kQ},\quad\text{for }p\in P\setminus P_j.\]
	Since $e^j_p$ ($p\in P$) are definable functions on $C_{j,t}$, by \eqref{epcong} we deduce that $S_{j,t,w}$ is definable in $\Pr$. By \eqref{Sprime}, we see that
	\[S^\prime_j=S^\prime\cap C_j=\bigsqcup_{t\in T}\bigsqcup_{w\in W_t}S_{j,t,w}\]
	is also definable in $\Pr$. Hence $S^\prime_j$ is semi-linear by Theorem~\ref{thmGS}.
\end{proof}

\begin{Cor}\label{polypower}
	Let $d_1,d_2\geq2$ and $d_3\geq0$ be integers, and let $\zeta,\eta\in U(\C)$. Define $f(z)=z^{d_1}$, $g(z)=\zeta z^{d_2}$, and $c(z)=\eta z^{d_3}$. Here $c=\eta$ if $d_3=0$. Then
	\[S_{f,g,c}^2=\left\lbrace(m,n)\in\Z_{\geq0}^2\colon\exists\la\in\C,\,f^{\circ m}(\la)=g^{\circ n}(\la)=c(\la)\right\rbrace\]
	is semi-linear.
\end{Cor}
\begin{proof}
	Suppose first that $d_3\geq1$. Then for every $(m,n)\in\Z_{\geq0}^2$, we have
	\[f^{\circ m}(0)=g^{\circ n}(0)=c(0)=0.\]
	Thus, $S_{f,g,c}^2=\Z_{\geq0}^2$ is linear.
	
	Suppose now $d_3=0$; then $c=\eta$ is a root of unity. Since $c\in\C^\times$, we deduce that $S_{f,g,c}^2=\hat{S}_{f,g,c}^2$ and $S_{f,g,c}^2$ is semi-linear by Proposition~\ref{powerconst}.
\end{proof}

\medskip

Next we proceed to deal with Chebyshev polynomials.
Let $\nu_2:\Q\to\Z\cup\{\infty\}$ be the usual $2$-adic valuation on $\Q$.
\begin{Prop}\label{powerpm1not0}
	Let $r,s\geq2$ and $t\geq1$ be integers, and let $\varepsilon_1,\varepsilon_2\in\{\pm1\}$.
	Set $\mu_r=\nu_2(r)$, $\mu_s=\nu_2(s)$, and $\mu_t=\nu_2(t)$. Let $m_0\geq1$ be the minimal integer such that
	\[r^{m_0}>t\quad\text{and}\quad\mu_r(\mu_r m_0-\mu_t-2)\geq0,\]
	and let $n_0\geq1$ be the minimal integer such that
	\[s^{n_0}>t\quad\text{and}\quad\mu_s(\mu_s n_0-\mu_t-2)\geq0.\]
	Define $V(r,s,t,\varepsilon_1,\varepsilon_2)$ to be the set
	\[\left\lbrace (m,n)\in\Z_{\geq m_0}\times\Z_{\geq n_0}\colon\exists\la\in\C^\times,\exists\eta,\kappa\in\{\pm1\},\la^{r^m-\eta t}=\varepsilon_1\text{ and }\la^{s^n-\kappa t}=\varepsilon_2\right\rbrace.\]
	Then either $V(r,s,t,\varepsilon_1,\varepsilon_2)=\emptyset$, or $V(r,s,t,\varepsilon_1,\varepsilon_2)=\Z_{\geq m_0}\times\Z_{\geq n_0}$. In particular, $V(r,s,t,\varepsilon_1,\varepsilon_2)$ is semi-linear.
\end{Prop}
\begin{proof}
	If $\varepsilon_1=\varepsilon_2=1$, then taking $\la=1$ shows that $V(r,s,t,1,1)=\Z_{\geq m_0}\times\Z_{\geq n_0}$. Suppose now that $\varepsilon_1=-1$ or $\varepsilon_2=-1$.
	
	\medskip
	
	\textbf{Case (1).} Exactly one of $\varepsilon_1$ and $\varepsilon_2$ is $-1$.
	
	Without loss of generality, we may assume that $\varepsilon_1=1$ and $\varepsilon_2=-1$. Applying Lemma~\ref{lem4.1} with $(k,\xi,a,b)=(2,-1,0,1)$, for every $(m,n)\in\Z_{\geq m_0}\times\Z_{\geq n_0}$, we have
	\begin{equation}\label{case1}
		(m,n)\in V(r,s,t,1,-1)\iff\exists\eta,\kappa\in\{\pm1\},\,\nu_2(s^n-\kappa t)<\nu_2(r^m-\eta t).
	\end{equation}
	
	If $r\not\equiv t\Mod{2}$, then $\nu_2(r^m-\eta t)=0$ for all $(m,n)\in\Z_{\geq m_0}\times\Z_{\geq n_0}$ and $\eta\in\{\pm1\}$; by \eqref{case1} we deduce that $V(r,s,t,1,-1)=\emptyset$. Suppose now that $r\equiv t\Mod{2}$.
	
	\smallskip
	
	\textbf{Case (1.1).} $(\varepsilon_1,\varepsilon_2)=(1,-1)$, and both $r$ and $t$ are even.
	
	If $s$ is odd, then
	\[\nu_2(s^n-t)=0<1\leq\nu_2(r^m-t)\]
	for all $(m,n)\in\Z_{\geq m_0}\times\Z_{\geq n_0}$. It follows from \eqref{case1} that
	\[V(r,s,t,1,-1)=\Z_{\geq m_0}\times\Z_{\geq n_0},\]
	which is semi-linear. Suppose now that $s$ is even.
	
	Then $r$, $t$, and $s$ are all assumed to be even. Let $(m,n)\in\Z_{\geq m_0}\times\Z_{\geq n_0}$ be arbitrary. Since
	\[\mu_r=\nu_2(r)\geq1\quad\text{and}\quad\mu_s=\nu_2(s)\geq1,\]
	we have
	\begin{equation}\label{mut+2}
		\min\{\nu_2(r^m),\nu_2(s^n)\}\geq\min\left\lbrace\mu_r m_0,\mu_s n_0\right\rbrace\geq\mu_t+2=\nu_2(t)+2
	\end{equation}
	by the definition of $(m_0,n_0)$. Then for all $\eta,\kappa\in\{\pm1\}$, we obtain
	\begin{equation}\label{case11}
		\nu_2(s^n-\kappa t)=\nu_2(r^m-\eta t)=\nu_2(t)=\mu_t.
	\end{equation}
	By \eqref{case1} and \eqref{case11}, we deduce that $V(r,s,t,1,-1)=\emptyset$.
	
	\smallskip
	
	\textbf{Case (1.2).} $(\varepsilon_1,\varepsilon_2)=(1,-1)$, and both $r$ and $t$ are odd.
	
	If $s$ is even, then for all $(m,n)\in\Z_{\geq m_0}\times\Z_{\geq n_0}$ we have
	\[\nu_2(s^n-t)=0<1\leq\nu_2(r^m-\eta t);\]
	by \eqref{case1} we deduce that $V(r,s,t,1,-1)=\Z_{\geq m_0}\times\Z_{\geq n_0}$. Suppose now $s$ is odd.
	
	Then $r$, $t$, and $s$ are all assumed to be odd. Let $(m,n)\in\Z_{\geq m_0}\times\Z_{\geq n_0}$ be arbitrary. Working modulo $4$, we can take $\eta_0,\kappa_0\in\{\pm1\}$ (that depend on $m$ and $n$) such that
	\[s^n-\kappa_0 t\equiv2\Mod{4}\quad\text{and}\quad r^m-\eta_0 t\equiv0\Mod{4},\]
	hence
	\begin{equation}\label{case12}
		\nu_2(s^n-\kappa_0 t)=1<2\leq\nu_2(r^m-\eta_0 t).
	\end{equation}
	By \eqref{case1} and \eqref{case12}, we deduce that $V(r,s,t,1,-1)=\Z_{\geq m_0}\times\Z_{\geq n_0}$.
	
	\medskip
	
	\textbf{Case (2).} $\varepsilon_1=\varepsilon_2=-1$.
	
	For every $(m,n)\in\Z_{\geq m_0}\times\Z_{\geq n_0}$, applying Lemma~\ref{lem4.1} with $(k,\xi,a,b)=(2,-1,1,1)$, the condition that $(m,n)\in V(r,s,t,-1,-1)$ is equivalent to
	\begin{equation}\label{case2}
		\exists\eta,\kappa\in\{\pm1\},\,\min\left\lbrace \nu_2(s^n-\kappa t),\nu_2(r^m-\eta t)\right\rbrace<\nu_2(r^m-s^n-(\eta-\kappa)t).
	\end{equation}
	
	\smallskip
	
	\textbf{Case (2.1).} $(\varepsilon_1,\varepsilon_2)=(-1,-1)$, and $u\not\equiv t\Mod{2}$ for some $u\in\{r,s\}$.
	
	Without loss of generality, we may assume that $s\not\equiv t\Mod{2}$. Then for all $(m,n)\in\Z_{\geq m_0}\times\Z_{\geq n_0}$ and all $\eta,\kappa\in\{\pm1\}$, we have
	\begin{equation}\label{case21}
		\min\left\lbrace \nu_2(s^n-\kappa t),\nu_2(r^m-\eta t)\right\rbrace=\nu_2(s^n-\kappa t)=0.
	\end{equation}
	
	If $r\equiv s\Mod{2}$, then for all $(m,n)\in\Z_{\geq m_0}\times\Z_{\geq n_0}$ we have $\nu_2(r^m-s^n)\geq1$; by \eqref{case2} (with $\eta=\kappa$) and \eqref{case21}, we deduce that $V(r,s,t,-1,-1)=\Z_{\geq m_0}\times\Z_{\geq n_0}$ which is semi-linear. Suppose now that $r\not\equiv s\Mod{2}$.
	
	Then for all $(m,n)\in\Z_{\geq m_0}\times\Z_{\geq n_0}$ and all $\eta,\kappa\in\{\pm1\}$, we have
	\[\nu_2(r^m-s^n-(\eta-\kappa)t)=0;\]
	by \eqref{case2}, we deduce that $V(r,s,t,-1,-1)=\emptyset$.
	
	\smallskip
	
	\textbf{Case (2.2).} $(\varepsilon_1,\varepsilon_2)=(-1,-1)$, and $r\equiv s\equiv t\Mod{2}$.
	
	Suppose first that $r$, $s$, and $t$ are all even. Let $(m,n)\in\Z_{\geq m_0}\times\Z_{\geq n_0}$ be arbitrary. Clearly, we can take $\eta_0\in\{\pm1\}$ (that depends on $m$) such that
	\begin{equation}\label{case221}
		\nu_2(r^m-\eta_0t)\leq\nu_2(2t)=\mu_t+1.
	\end{equation}
	Since $\mu_r,\mu_t\geq1$, as in Case (1.1) we have $\min\{\nu_2(r^m),\nu_2(s^n)\}\geq\mu_t+2$ (see \eqref{mut+2}). Then
	\begin{equation}\label{case221p}
		\nu_2(r^m-s^n)\geq\min\{\nu_2(r^m),\nu_2(s^n)\}>\mu_t+1.
	\end{equation}
	By \eqref{case2}, \eqref{case221}, and \eqref{case221p}, we conclude that $V(r,s,t,-1,-1)=\Z_{\geq m_0}\times\Z_{\geq n_0}$.
	
	Suppose now that $r$, $s$, and $t$ are all odd. Let $(m,n)\in\Z_{\geq m_0}\times\Z_{\geq n_0}$ be arbitrary. Working modulo $4$, we can take $\eta_0,\kappa_0\in\{\pm1\}$ (that depend on $m$ and $n$) such that
	\[s^n-\kappa_0 t\equiv r^m-\eta_0 t\equiv2\Mod{4},\]
	hence
	\begin{equation}\label{case222}
		\nu_2(s^n-\kappa_0 t)=\nu_2(r^m-\eta_0 t)=1\quad\text{and}\quad\nu_2(r^m-s^n-(\eta_0-\kappa_0)t)\geq2.
	\end{equation}
	By \eqref{case2} and \eqref{case222}, we deduce that $V(r,s,t,-1,-1)=\Z_{\geq m_0}\times\Z_{\geq n_0}$.
\end{proof}

\begin{Prop}\label{Cheby}
	Let $r,s\geq2$ and $t\geq1$ be integers, and let $\varepsilon_1,\varepsilon_2,\varepsilon_3\in\{\pm1\}$. Define
	\[f(z)=\varepsilon_1 T_r(z),\quad g(z)=\varepsilon_2 T_s(z),\quad\text{and}\quad c(z)=\varepsilon_3 T_t(z),\]
	where we set $T_1(z)=z$. Then $S_{f,g,c}^2$ is semi-linear.
\end{Prop}
\begin{proof}
	Set $\mu_r=\nu_2(r)$, $\mu_s=\nu_2(s)$, and $\mu_t=\nu_2(t)$. Let $m_0,n_0\geq1$ be as given in Proposition~\ref{powerpm1not0}. By Lemma~\ref{lemSs}, it suffices to prove that
	\[S^\prime:=S_{f,g,c}^2\cap\left(\Z_{\geq m_0}\times\Z_{\geq n_0}\right)\]
	is semi-linear.
	
	Let $\pi:\P^1_\C\to\P^1_\C$ be the endomorphism given by $\pi(z)=z+z^{-1}$. Note that for every $\delta\in\C$, the preimage set $\pi^{-1}(\delta)$ equals $\{\la,\la^{-1}\}$ for some $\la\in\C^\times$, and conversely, $\pi(\la)=\pi(\la^{-1})\in\C$ for every $\la\in\C^\times$. Define
	\[F(z)=\varepsilon_1 z^{r},\quad G(z)=\varepsilon_2 z^{s},\quad\text{and}\quad h(z)=\varepsilon_3 z^{t}.\]
	Let $\omega_1,\omega_2\in\{0,1\}$ such that
	\[r\equiv \omega_1\Mod{2}\quad\text{and}\quad s\equiv\omega_2\Mod{2}.\]
	For $(m,n)\in\Z_{\geq m_0}\times\Z_{\geq n_0}$, by \eqref{piTd} we obtain
	\begin{align*}
		&(m,n)\in S^\prime\text{, i.e., }\exists\delta\in\C,\,f^{\circ m}(\delta)=g^{\circ n}(\delta)=c(\delta)\\
		\iff&\exists\la\in\C^\times,\,\pi\circ F^{\circ m}(\la)=\pi\circ G^{\circ n}(\la)=\pi\circ h(\la)\\
		\iff&\exists\la\in\C^\times,\,\exists\eta,\kappa\in\{\pm1\},\,F^{\circ m}(\la)=h(\la)^\eta\text{ and }G^{\circ n}(\la)=h(\la)^\kappa\\
		\iff&\exists\la\in\C^\times,\,\exists\eta,\kappa\in\{\pm1\},\,\la^{r^m-\eta t}=\varepsilon_1^{1+\omega_1(m-1)}\varepsilon_3\text{ and }\la^{s^n-\kappa t}=\varepsilon_2^{1+\omega_2(n-1)}\varepsilon_3.
	\end{align*}
	
	Set
	\[J=\{(0,0),(0,1),(1,0),(1,1)\}.\]
	Let $j=(j_1,j_2)\in J$ be arbitrary. Define
	\[Z_j=\left(\Z_{\geq m_0}\times\Z_{\geq n_0}\right)\cap\left(\left(2\Z_{\geq0}+j_1\right)\times\left(2\Z_{\geq0}+j_2\right)\right)\quad\text{and}\quad S^\prime_j=S_{f,g,c}^2\cap Z_j.\]
	Define
	\[\varepsilon_{1j}=\varepsilon_1^{1+\omega_1(j_1-1)}\varepsilon_3\in\{\pm1\}\quad\text{and}\quad \varepsilon_{2j}=\varepsilon_2^{1+\omega_2(j_2-1)}\varepsilon_3\in\{\pm1\}.\]
	Then we obtain
	\begin{align*}
		S^\prime_j&=\left\lbrace (m,n)\in Z_j\colon \exists\la\in\C^\times,\,\exists\eta,\kappa\in\{\pm1\},\,\la^{r^m-\eta t}=\varepsilon_{1j}\text{ and }\la^{s^n-\kappa t}=\varepsilon_{2j}\right\rbrace\\
		&=V(r,s,t,\varepsilon_{1j},\varepsilon_{2j})\cap Z_j,
	\end{align*}
	where $V(r,s,t,\varepsilon_{1j},\varepsilon_{2j})$ is as in Proposition~\ref{powerpm1not0}. Clearly, $Z_j$ is definable in $\Pr$. Hence we conclude that $S^\prime_j=V(r,s,t,\varepsilon_{1j},\varepsilon_{2j})\cap Z_j$ is semi-linear by Theorem~\ref{thmGS} and Proposition~\ref{powerpm1not0}. Therefore,
	\[S^\prime=\bigsqcup_{j\in J}S^\prime_j\]
	is also semi-linear. The proof is finished.
\end{proof}

At the end of this subsection, we prove Proposition~\ref{powertil} on the modified recurrence set for certain power maps. We need the following elementary lemma, which is the $2$-adic case of the folklore lifting-the-exponent (LTE) lemma. See \cite{LTE} for a proof.
\begin{Lem}[$2$-adic LTE lemma]\label{2LTE}
Let $x,y$ be odd integers, and let $n\geq1$ be a positive integer.
\begin{enumerate}
\item If $n$ is even, then $\nu_2(x^n-y^n)=\nu_2(n)+\nu_2(x-y)+\nu_2(x+y)-1$ and $\nu_2(x^n+y^n)=1$.
\item If $n$ is odd, then $\nu_2(x^n-y^n)=\nu_2(x-y)$ and $\nu_2(x^n+y^n)=\nu_2(x+y)$.
\end{enumerate} 
\end{Lem}

\begin{proof}[Proof of Proposition~\ref{powertil}]
	Let
	\[f(z)=z^r,\quad g(z)=z^s,\quad\text{and}\quad c(z)=-z,\]
	where $r,s\geq3$ are odd integers. Our goal is to prove that
	\[\tilde{S}_{f,g,c}^2:=\left\lbrace (m,n)\in\Z_{\geq0}^2\colon \exists\la\in\C^\times,\,f^{\circ m}(\la)=g^{\circ n}(\la)=c(\la)\right\rbrace\]
	is not semi-linear.
	
	Suppose, for the sake of contradiction, that $\tilde{S}_{f,g,c}^2$ is semi-linear. By Theorem~\ref{thmGS}, the set
	\[S^\prime:=\left\lbrace (m,n)\in\Z_{>0}^2\colon (2m,2n)\in\tilde{S}_{f,g,c}^2\right\rbrace\]
	is also semi-linear. Let $(m,n)\in\Z_{>0}^2$. Applying Lemma~\ref{lem4.1} with $k=2$, we deduce that
	\begin{align*}
		&(m,n)\in S^\prime\text{, i.e., }\exists\la\in\C^\times,\,f^{\circ 2m}(\la)=g^{\circ 2n}(\la)=c(\la)\\
		\iff&\min\left\lbrace\nu_2(r^{2m}-1),\nu_2(s^{2n}-1)\right\rbrace<\nu_2(r^{2m}-s^{2n}).
	\end{align*}
	If $\nu_2(r^{2m}-1)\neq\nu_2(s^{2n}-1)$, then the strong triangle inequality gives
	\[\min\left\lbrace \nu_2(r^{2m}-1),\nu_2(s^{2n}-1)\right\rbrace=\nu_2(r^{2m}-s^{2n}).\]
	If $\nu_2(r^{2m}-1)=\nu_2(s^{2n}-1)=:\mu$, then by modulo $2^{\mu+1}$ we obtain that
	\[\nu_2(r^{2m}-s^{2n})=\nu_2((r^{2m}-1)-(s^{2n}-1))\geq\mu+1.\]
	By Lemma~\ref{2LTE}, for every $(m,n)\in\Z_{>0}^2$, we have
	\[\nu_2(r^{2m}-1)-\nu_2(s^{2n}-1)=\nu_2(m)-\nu_2(n)+a_{r,s},\]
	where the integer
	\[a_{r,s}:=\nu_2(r-1)+\nu_2(r+1)-\nu_2(s-1)-\nu_2(s+1)\in\Z\]
	depends only on $r$ and $s$. 
	Thus, we conclude that
	\[S^\prime=\left\lbrace(m,n)\in\Z_{>0}^2\colon\nu_2(m)+a_{r,s}=\nu_2(n)\right\rbrace.\]
	
	By Lemma~\ref{uniD}, we can take $D\in\Z_{>0}$ such that $\ep(S^\prime_m)\leq D$ for all $m\geq1$. Take an integer $N>\max\left\lbrace 0,-a_{r,s}+\log_2 D\right\rbrace $. Then
	\[N^\prime:=N+a_{r,s}>\log_2 D\] and we have
	\[S^\prime_{2^N}=\left\lbrace n\in\Z_{>0}\colon \nu_2(n)=N^\prime\right\rbrace =2^{N^\prime}+2^{N^\prime+1}\Z_{\geq0}.\]
	A direct computation shows that $\ep(S^\prime_{2^N})=2^{N^\prime+1}>2D>D$, which is a contradiction. Therefore, the modified rank-two recurrence set $\tilde{S}_{f,g,c}^2$ is not semi-linear. 
\end{proof}

\begin{Rem}
	Let $(f,g,c)$ be as in Proposition~\ref{powertil}. Modifying the proof of Proposition~\ref{powertil}, we can deduce that $\tilde{S}_{f,g,c}^2$ is definable in the structure
	\[{\mathrm{B\ddot{u}}}_2:=(\Z_{\geq0},+,0,1,V_2),\]
	where $V_2:\Z_{\geq0}\to\Z_{\geq0}$ is the function of one variable given by
	\[V_2(m)=2^{\nu_2(m)}\text{ for }m\in\Z_{>0}\quad\text{and}\quad V_2(0)=1.\]
	This structure $\mathrm{B\ddot{u}}_2$ is often called the \emph{B\"uchi arithmetic of base $2$}. See \cite{Buchi} for more information.
	
	It is natural to ask whether all modified rank-two recurrence sets for power maps are definable in the structure
	\[\left(\Z_{\geq0},+,0,1,(V_p)_{p\text{ prime}}\right),\]
	where $V_p$ is defined similarly as above for a prime $p\in\Z$.
\end{Rem}

\subsection{Rank-two recurrence for polynomials}\label{secrk2poly}
We first deal with the case where exactly one of $f$ and $g$ is of degree $1$.
\begin{Prop}\label{onedeg1}
	Let $f,g,c\in\C[z]$ be polynomials such that $\deg(f)\geq2$ and $\deg(g)=1$. Then $S_{f,g,c}^2$ is semi-linear.
\end{Prop}
\begin{proof}
	Define
	\[A_f=\left\lbrace m\in\Z_{\geq0}\colon f^{\circ m}=c\right\rbrace\quad\text{and}\quad A_g=\left\lbrace n\in\Z_{\geq0}\colon g^{\circ n}=c\right\rbrace.\]
	Then $\# A_f\leq1$ because $\deg(f)\geq2$. Define
	\[S_s^f=S_{f,g,c}^2\cap\left(A_f\times\Z_{\geq0}\right)\quad\text{and}\quad S_s^g=S_{f,g,c}^2\cap\left(\Z_{\geq0}\times A_g\right).\]
	For every $\la\in\C$, define
	\[S^2(\la):=\left\lbrace (m,n)\in\Z_{\geq0}^2\colon f^{\circ m}(\la)=g^{\circ n}(\la)=c(\la)\right\rbrace.\]
	Then $S^2(\la)=S_{f,c}(\la)\times S_{g,c}(\la)$ is semi-linear for all $\la\in\C$; see \eqref{Shla}.
	
	By \cite[Proposition~9]{HT17}, we can take a finite subset $\La\subseteq\C$ such that
	\begin{equation}\label{HTdecomp}
		S_{f,g,c}^2=S_s^f\cup S_s^g\cup\bigcup_{\la\in\La}S^2(\la).
	\end{equation}
	With the help of the decomposition \eqref{HTdecomp}, it suffices to prove that both $S_s^f$ and $S_s^g$ are semi-linear.
	
	\smallskip
	
	As $\deg(f)\geq2$, we have
	\[S_s^g=Z_f\times A_g,\]
	where $Z_f=\Z_{\geq0}$ or $Z_f=\Z_{\geq0}\setminus\{\log_{\deg(f)}\deg(c)\}$; see the proof of Lemma~\ref{lemSs}. In particular, $\#\left(\Z_{\geq0}\setminus Z_f\right)\leq1$ and $Z_f$ is semi-linear. If $\# A_g<\infty$, then $S_s^g=Z_f\times A_g$ is semi-linear. If $\# A_g=\infty$, then it is easy to see that
	\[A_g=\left\lbrace n_0+kn\colon n\in\Z_{\geq0}\right\rbrace\]
	is an arithmetic progression, where $n_0=\min A_g$ and $k=-n_0+\min\left(A_g\setminus\{n_0\}\right)$; hence, $S_s^g=Z_f\times A_g$ is semi-linear. We have shown that $S_s^g$ is always semi-linear. It remains to show that $S_s^f$ is semi-linear. 
	
	\smallskip
	
	If $A_f=\emptyset$, then $S_s^f=\emptyset$ is semi-linear. Suppose now that $A_f\neq\emptyset$. Then $A_f=\{m_0\}$ for some $m_0\geq0$. If $m_0>0$, then
	\[\deg(g^{\circ n})=1<\deg(f^{\circ m_0})=\deg(c)\]
	for all $n\geq0$; hence, we conclude that $S_s^f=\{m_0\}\times\Z_{\geq0}$ is semi-linear. Suppose now that $m_0=0$. Then $c(z)=f^{\circ 0}(z)=z$. Write
	\[g(z)=az+b\]
	with $a\in\C^\times$ and $b\in\C$. Then $g^{\circ 0}(z)=z=c(z)$, and for every $n\in\Z_{>0}$, we have
	\[g^{\circ n}(z)=a^n z+b\sum_{j=0}^{n-1}a^j.\]
	\begin{equation}\label{Ssf}
		S_s^f=\{(0,0)\}\cup\left\lbrace (0,n)\colon n\in\Z_{>0}\text{ such that }a^n\neq1\text{ or }b\sum_{j=0}^{n-1}a^j=0\right\rbrace.
	\end{equation}
	Using \eqref{Ssf}, we compute that $S_s^f=\{0\}\times\Z_{\geq0}$ if $a\neq1$ or $b=0$, and that $S_s^f=\{(0,0)\}$ if $a=1$ and $b\neq0$. Therefore, $S_s^f$ is always semi-linear. The proof is complete.
\end{proof}

We are now ready to prove Theorem~\ref{thmrk2poly}.
\begin{proof}[Proof of Theorem~\ref{thmrk2poly}]
	Let $f,g\in\C[z]\setminus\C$ be non-constant polynomials such that $d_1:=\deg(f)\geq2$ or $d_2:=\deg(g)\geq2$, and let $c\in\C[z]$. By symmetry and Proposition~\ref{onedeg1}, we may assume that $d_1\geq2$ and $d_2\geq2$. By Proposition~\ref{Rfinite}, we may assume that $\# R_{f,g,c}=\infty$. Then Proposition~\ref{poly2musame} implies that
	\[P:=\PrePer(f,\C)=\PrePer(g,\C),\quad\mu:=\mu_f=\mu_g,\quad\text{and}\quad J:=J(f)=J(g).\]
	Since $\mu_f=\mu_g$, by \cite{Zdunik90}, there are three possibilities:
	\begin{itemize}
		\item both $f$ and $g$ are non-exceptional;
		\item $f$ is conjugate to $z^{d_1}$ and $g$ is conjugate to $z^{d_2}$;
		\item $f$ is conjugate to $\pm T_{d_1}$ and $g$ is conjugate to $\pm T_{d_2}$.
	\end{itemize}
	
	\medskip
	
	\textbf{Case 1.} Suppose that both $f$ and $g$ are non-exceptional.
	
	By Theorem~\ref{polysamemu} and Remark~\ref{AutJ}, after affine conjugacy, we may assume that
	\begin{equation}\label{fgform2}
		f(z)=\zeta_1 h^{\circ k_1}(z)\quad\text{and}\quad g(z)=\zeta_2 h^{\circ k_2}(z),
	\end{equation}
	where $k_1,k_2\geq1$ are integers, and $h\in\C[z]$ is a monic centered non-exceptional polynomial of degree $\geq2$ with $\mu_h=\mu$ and $J(h)=J$ of the form
	\begin{equation}\label{hform2}
		h(z)=z^r R(z^s)\quad(r\geq1,\,s\geq1,\,\text{and }R\in\C[z]\setminus\C)
	\end{equation}
	as in Remark~\ref{AutJ}, and $\zeta_1,\zeta_2\in U_s(\C)$. Set $d:=\deg(h)\geq2$. Then $d_1=d^{k_1}$ and $d_2=d^{k_2}$. From \eqref{hform2} and \eqref{fgform2}, it is easy to see that
	\[\PrePer(h,\C)=P\quad\text{and}\quad P=U_s(\C)\cdot P=\left\lbrace\zeta z_0\colon\zeta\in U_s(\C)\text{ and }z_0\in P\right\rbrace.\]
	We now consider the following two subcases.
	
	\smallskip
	
	\textbf{Case 1.1.} Suppose that $c\in\C$.
	
	Then $c\in P$ by Proposition~\ref{poly3cases}.
	
	By Lemma~\ref{lemSs}, it suffices to prove that
	\[S^\prime:=S_{f,g,c}^2\cap\Z_{>0}^2\]
	is semi-linear. Consider the partition $\Z_{>0}^2=C_1\sqcup C_2$, where
	\[C_1:=\left\lbrace (m,n)\in\Z_{>0}^2\colon k_1m\geq k_2n\right\rbrace\quad\text{and}\quad C_2:=\left\lbrace (m,n)\in\Z_{>0}^2\colon k_1m<k_2n\right\rbrace.\]
	Define $S^\prime_j:=S^\prime\cap C_j$ for $1\leq j\leq 2$. Then $S^\prime=S^\prime_1\sqcup S^\prime_2$, and it remains to prove that both $S^\prime_1$ and $S^\prime_2$ are semi-linear. In the following, we will only prove that $S^\prime_1$ is semi-linear, because the semi-linearity of $S^\prime_2$ can be proved similarly.
	
	For $m\in\Z_{>0}$, since $r\equiv d\Mod{s}$, we compute that
	\[f^{\circ m}(z)=\xi_0^{a(d^{k_1 m}-1)/(d^{k_1}-1)}h^{\circ k_1m}(z)\quad\text{and}\quad g^{\circ m}(z)=\xi_0^{b(d^{k_2 m}-1)/(d^{k_2}-1)}h^{\circ k_2m}(z).\]
	
	Let $(m,n)\in C_1$ be arbitrary. Observe that the equation
	\[\xi_0^{b(d^{k_2 n}-1)/(d^{k_2}-1)}h^{\circ k_2n}(\la)=c\quad(\la\in\C)\]
	always has a solution in $\C$ because $\deg\left(h^{\circ k_2n}\right)>0$. We deduce that
	\begin{align*}
		&(m,n)\in S^\prime_1\\
		\iff&\exists\la\in\C,\,\xi_0^{a(d^{k_1 m}-1)/(d^{k_1}-1)}h^{\circ k_1m}(\la)=\xi_0^{b(d^{k_2 n}-1)/(d^{k_2}-1)}h^{\circ k_2n}(\la)=c\\
		\iff&\xi_0^{a(d^{k_1 m}-1)/(d^{k_1}-1)}h^{\circ (k_1m-k_2n)}\left(\xi_0^{-b(d^{k_2 n}-1)/(d^{k_2}-1)}c\right)=c\\
		\iff&h^{\circ (k_1m-k_2n)}(c)=\xi_0^{\varphi(m,n)}c,
	\end{align*}
	where $\varphi:C_1\to\Z$ is the function given by
	\[\varphi(m,n)=-a\frac{d^{k_1m}-1}{d^{k_1}-1}+bd^{k_1m-k_2n}\frac{d^{k_2n}-1}{d^{k_2}-1}.\]
	
	Suppose first that $c=0$. Then we have
	\[S^\prime_1=\left\lbrace (m,n)\in C_1\colon h^{\circ (k_1m-k_2n)}(0)=0\right\rbrace.\]
	Recall that
	\[S_{h,0}(0):=\left\lbrace m\in\Z_{\geq0}\colon h^{\circ m}(0)=0\right\rbrace\]
	is semi-linear; see \eqref{Shla}. Observe that $C_1$ is definable in $\Pr$, and the function $(m,n)\mapsto k_1m-k_2n$ on $C_1$ is also definable in $\Pr$. By Theorem~\ref{thmGS}, we conclude that
	\[S^\prime_1=\left\lbrace (m,n)\in C_1\colon k_1m-k_2n\in S_{h,0}(0)\right\rbrace\]
	is definable in $\Pr$ and hence semi-linear.
	
	\smallskip
	
	Suppose now that $c\neq0$. Define
	\[\Omega:=\left\lbrace\xi c\colon\xi\in U_s(\C)\right\rbrace\cap O_h(c).\]
	Then $\Omega$ is a finite set and it may be empty. We deduce that
	\begin{equation}\label{11}
		S^\prime_1=\left\lbrace (m,n)\in C_1\colon\exists c_0\in\Omega,h^{\circ (k_1m-k_2n)}(c)=c_0\text{ and }c_0=\xi_0^{\varphi(m,n)}c\right\rbrace.
	\end{equation}
	
	Fix an arbitrary $c_0\in\Omega$. The definition of $\Omega$ implies that
	\begin{equation}\label{Shc0}
		S_{h,c_0}(c)=\left\lbrace m\in\Z_{\geq0}\colon h^{\circ m}(c)=c_0\right\rbrace\text{ is an arithmetic progression;}
	\end{equation}
	see also \eqref{Shla}. Let $l(c_0)\in\{0,1,\dots,s-1\}$ be the unique integer such that
	\[c_0=\xi_0^{l(c_0)}c.\]
	Then for every $(m,n)\in C_1$, we have
	\[c_0=\xi_0^{\varphi(m,n)}c\iff l(c_0)\equiv \varphi(m,n)\Mod{s}.\]
	
	Note that for every $D\in\Z_{>0}$, the sequence
	\[\left(\frac{D^n-1}{D-1}\Mod{s}\right)_{n=0}^\infty\]
	is eventually periodic; hence, for every $D\in\Z_{>0}$ and every $q\in\Z$, the set
	\begin{equation}\label{epcong2}
		\left\lbrace n\in\Z_{\geq0}\colon\frac{D^n-1}{D-1}\equiv q\Mod{s}\right\rbrace \text{ is semi-linear.}
	\end{equation}
	Consider the congruence equation
	\begin{equation}\label{cong2}
		-ax+byw\equiv l(c_0)\Mod{s}
	\end{equation}
	in variables $(x,y,w)$. Let
	\[T(c_0)\subseteq\{0,1,\dots,s-1\}^3\]
	be the set of all solutions of \eqref{cong2} over the least residue system $\{0,1,\dots,s-1\}$ of $s$. Let $t_0=(x_0,y_0,w_0)\in T(c_0)$ be arbitrary. Define $S^\prime_{1,c_0,t_0}$ to be the set of all pairs $(m,n)\in C_1$ satisfying
	\[\frac{d^{k_1m}-1}{d^{k_1}-1}\equiv x_0\Mod{s},\quad d^{k_1m-k_2n}\equiv y_0\Mod{s},\]
	\[\frac{d^{k_2n}-1}{d^{k_2}-1}\equiv w_0\Mod{s},\quad\text{and}\quad k_1m-k_2n\in S_{h,c_0}(c).\]
	Since the function $(m,n)\mapsto k_1m-k_2n$ is definable on $C_1$, from \eqref{Shc0}, \eqref{epcong}, and \eqref{epcong2}, we deduce that $S^\prime_{1,c_0,t_0}$ is definable in $\Pr$.
	
	Therefore,
	\[S^\prime_1=\bigsqcup_{c_0\in\Omega}\bigsqcup_{t_0\in T(c_0)}S^\prime_{1,c_0,t_0}\]
	is definable in $\Pr$. By Theorem~\ref{thmGS}, we conclude that $S^\prime_1$ is semi-linear. The proof of Case 1.1 is complete.
	
	\smallskip
	
	\textbf{Case 1.2.} Suppose that $d_3=\deg(c)\geq1$.
	
	By Proposition~\ref{poly3cases}, Theorem~\ref{polysamemu} and Remark~\ref{AutJ} imply that $c$ is of the form
	\[c(z)=\zeta_3 h^{\circ k_3}(z)\]
	for some $\zeta_3\in U_s(\C)$ and $k_3\in\Z_{\geq0}$.
	
	If $r>0$, then
	\[f(0)=g(0)=c(0)=h(0)=0;\]
	see \eqref{hform2}. Taking $\la=0$ implies that $S_{f,g,c}^2=\Z_{\geq0}^2$ which is semi-linear, in the case $r>0$.
	
	Suppose now that $r=0$. Then $h(z)=R(z^s)$ and
	\[h(z)=h(\xi z)\quad\text{for all }\xi\in U_s(\C).\]
	For all $m,n\in\Z_{>0}$, we compute that
	\[f^{\circ m}(z)=\zeta_1 h^{\circ k_1m}(z)\quad\text{and}\quad g^{\circ n}(z)=\zeta_2 h^{\circ k_2n}(z).\]
	
	Let $m_0\geq1$ (resp. $n_0\geq1$) be the minimal integer such that $k_1 m_0\geq k_3+1$ (resp. $k_2 n_0\geq k_3+1$). By Lemma~\ref{lemSs}, it suffices to prove that
	\[S^\prime:=S_{f,g,c}^2\cap\left(\Z_{\geq m_0}\times\Z_{\geq n_0}\right)\]
	is semi-linear. Note that for every $(m,n)\in\Z_{\geq m_0}\times\Z_{\geq n_0}$, we have
	\[\min\left\lbrace k_1m,k_2n \right\rbrace\geq k_3+1.\]
	We consider two sub-subcases as follows.
	
	\smallskip
	
	\textbf{Case 1.2.1.} Suppose that $\zeta_1=\zeta_2$.
	
	Denote $\zeta:=\zeta_1=\zeta_2$. As $\deg(h)\geq2$, we can take $y_0\in\C$ such that $h(y_0)=y_0$. Take a preimage
	\[\la_0\in \left(h^{\circ k_3}\right)^{-1}\left(\zeta\zeta_3^{-1}y_0\right).\]
	For every $m\in\Z_{\geq m_0}$, since $r=0$ and $k_1m-k_3\geq1$, we obtain
	\begin{align*}
		f^{\circ m}(\la_0)&=\zeta h^{\circ k_1m}(\la_0)=\zeta h^{\circ (k_1m-k_3)}\circ h^{\circ k_3}(\la_0)\\
		&=\zeta h^{\circ (k_1m-k_3)}\left(\zeta\zeta_3^{-1}y_0\right)=\zeta h^{\circ (k_1m-k_3)}(y_0)=\zeta y_0\\
		&=\zeta\left(\zeta^{-1}\zeta_3 h^{\circ k_3}(\la_0)\right)=\zeta_3 h^{\circ k_3}(\la_0)\\
		&=c(\la_0).
	\end{align*}
	Similarly, we have $g^{\circ n}(\la_0)=c(\la_0)$ for every $n\in\Z_{\geq n_0}$. Taking $\la=\la_0$ shows that
	\[S^\prime=\Z_{\geq m_0}\times\Z_{\geq n_0},\]
	which is semi-linear.
	
	\smallskip
	
	\textbf{Case 1.2.2.} Suppose that $\zeta_1\neq\zeta_2$.
	
	Define
	\[\La:=\left\lbrace \la\in\C\colon\exists(m,n)\in\Z_{\geq m_0}\times\Z_{\geq n_0},\,f^{\circ m}(\la)=g^{\circ n}(\la)=c(\la)\right\rbrace.\]
	
	\smallskip
	
	\textbf{Claim:} $\La\subseteq \left(h^{\circ k_3}\right)^{-1}(0)$.
	
	\textit{Proof of the claim.} Let $\la\in\La$ and $(m,n)\in\Z_{\geq m_0}\times\Z_{\geq n_0}$ be such that $f^{\circ m}(\la)=g^{\circ n}(\la)=c(\la)$, i.e.,
	\[\zeta_1 h^{\circ k_1m}(\la)=\zeta_2 h^{\circ k_2n}(\la)=\zeta_3 h^{\circ k_3}(\la).\]
	Set $y=h^{\circ k_3}(\la)$. By symmetry, we may assume that $k_1m\geq k_2 n$. Recall that $k_2n\geq k_3+1$ because $n\geq n_0$. Note that
	\begin{equation}\label{122n}
		h^{\circ (k_2n-k_3)}(y)=h^{\circ (k_2n-k_3)}\circ h^{\circ k_3}(\la)=h^{\circ k_2 n}(\la)=\zeta_2^{-1}\zeta_3 h^{\circ k_3}(\la)=\zeta_2^{-1}\zeta_3y.
	\end{equation}
	Similarly, we have
	\begin{equation}\label{122m}
		h^{\circ (k_1m-k_3)}(y)=\zeta_1^{-1}\zeta_3y.
	\end{equation}
	Since $r=0$, by \eqref{122n} and \eqref{122m} we obtain
	\begin{align*}
		h^{\circ (k_1m-k_2n)}(y)&=h^{\circ (k_1m-k_2n)}(\zeta_2^{-1}\zeta_3y)=h^{\circ (k_1m-k_2n)}\circ h^{\circ (k_2n-k_3)}(y)\\
		&=h^{\circ (k_1m-k_3)}(y)=\zeta_1^{-1}\zeta_3y.
	\end{align*}
	With the help of the above equation, by \eqref{122n}, \eqref{122m}, and the assumption that $r=0$, we obtain
	\begin{align*}
		\zeta_2^{-1}\zeta_3y&=h^{\circ (k_2n-k_3)}(y)=h^{\circ (k_2n-k_3)}(\zeta_1^{-1}\zeta_3 y)=h^{\circ (k_2n-k_3)}\circ h^{\circ (k_1m-k_2n)}(y)\\
		&=h^{\circ (k_1m-k_3)}(y)=\zeta_1^{-1}\zeta_3y.
	\end{align*}
	We conclude that $y=0$ because $\zeta_2^{-1}\zeta_3\neq\zeta_1^{-1}\zeta_3$. The claim is proved.
	
	\smallskip
	
	By the claim, for every $(m,n)\in\Z_{\geq m_0}\times\Z_{\geq n_0}$, since $\min\left\lbrace k_1m,k_2n\right\rbrace>k_3 $, we obtain
	\begin{align*}
		&(m,n)\in S^\prime\text{, i.e., }\exists\la\in\C,\,f^{\circ m}(\la)=g^{\circ n}(\la)=c(\la)\\
		\iff&\exists\la\in\left(h^{\circ k_3}\right)^{-1}(0),\,\zeta_1 h^{\circ k_1m}(\la)=\zeta_2 h^{\circ k_2n}(\la)=\zeta_3 h^{\circ k_3}(\la)\\
		\iff&h^{\circ (k_1m-k_3)}(0)=h^{\circ (k_2n-k_3)}(0)=0.
	\end{align*}
	Recall that
	\[S_{h,0}(0):=\left\lbrace m\in\Z_{\geq0}\colon h^{\circ m}(0)=0\right\rbrace\]
	is semi-linear; see \eqref{Shla}. By Theorem~\ref{thmGS}, we conclude that
	\[S^\prime=\{(m,n)\in\Z_{\geq m_0}\times\Z_{\geq n_0}\colon k_1m\in S_{h,0}(0)+k_3\text{ and }k_2n\in S_{h,0}(0)+k_3\}\]
	is definable in $\Pr$ and hence semi-linear.
	
	\medskip
	
	\textbf{Case 2.} Suppose that $f$ is conjugate to $z^{d_1}$ and $g$ is conjugate to $z^{d_2}$.
	
	As in the proof of Proposition~\ref{poly3cases}, after affine conjugacy, we may assume that
	\[f(z)=z^{d_1},\quad g(z)=\zeta z^{d_2},\quad\text{and}\quad c(z)=\eta z^{d_3},\]
	where $\zeta,\eta\in U(\C)$, and $d_3=\deg(c)\geq0$. (Note that $c$ cannot be $0$ because $\# R_{f,g,c}=\infty$.) Then the conclusion follows from Corollary~\ref{polypower}.
	
	\medskip
	
	\textbf{Case 3.} Suppose that $f$ is conjugate to $\pm T_{d_1}$ and $g$ is conjugate to $\pm T_{d_2}$.
	
	As in the proof of Proposition~\ref{poly3cases}, after affine conjugacy, we may assume that
	\[f(z)=\varepsilon_1 T_{d_1}(z)\quad\text{and}\quad g(z)=\varepsilon_2 T_{d_2}(z)\]
	for some $\varepsilon_1,\varepsilon_2\in\{\pm1\}$. Note that $P=\{\xi+\xi^{-1}\colon\xi\in U(\C)\}$.
	
	If $\deg(c)\geq1$, then as in the proof of Proposition~\ref{poly3cases}, we have
	\[c(z)=\varepsilon_3 T_t(z)\]
	for some $\varepsilon_3\in\{\pm1\}$, where $t:=\deg(c)\geq1$ and we set $T_1(z)=z$. Thus, $S_{f,g,c}^2$ is semi-linear by Proposition~\ref{Cheby}.
	
	Suppose now that $c\in\C$. Let $\pi:\P^1_\C\to\P^1_\C$ be the endomorphism given by $\pi(z)=z+z^{-1}$. Then $c\in P$ by Proposition~\ref{poly3cases}. Hence
	\[c=\zeta+\zeta^{-1}=\pi(\zeta)\]
	for some $\zeta\in U(\C)$. By Lemma~\ref{lemSs}, it suffices to prove that
	\[S^\prime=S_{f,g,c}^2\cap\Z_{>0}^2\]
	is semi-linear. Define
	\[F(z)=\varepsilon_1 z^{d_1},\quad G_1(z)=\varepsilon_2 z^{d_2},\quad\text{and}\quad G_{-1}(z)=\varepsilon_2 z^{-d_2}.\]
	For $(m,n)\in\Z_{>0}^2$, we obtain
	\begin{align*}
		&(m,n)\in S_{f,g,c}^2\text{, i.e., }\exists\delta\in\C,\,f^{\circ m}(\delta)=g^{\circ n}(\delta)=c\\
		\iff&\exists\la\in\C^\times,\,\pi\circ F^{\circ m}(\la)=\pi\circ G^{\circ n}(\la)=\pi(\zeta)\\
		\iff&\exists\la\in\C^\times,\,\exists\eta,\kappa\in\{\pm1\},\,F^{\circ m}(\la)=\left(G^{\circ n}(\la)\right)^\eta=\zeta^\kappa\\
		\iff&\exists\eta,\kappa\in\{\pm1\},\,\exists\la\in\P^1(\C),\,F^{\circ m}(\la)=G_{\eta}^{\circ n}(\la)=\zeta^\kappa\\
		\iff&\exists\eta,\kappa\in\{\pm1\},\,(m,n)\in\hat{S}_{F,G_{\eta},\zeta^\kappa}^2.
	\end{align*}
	Therefore, we conclude that
	\[S^\prime=\left(\bigcup_{\eta,\kappa\in\{\pm1\}}\hat{S}_{F,G_{\eta},\zeta^\kappa}^2\right)\cap\Z_{>0}^2\]
	is semi-linear by Proposition~\ref{powerconst} and Theorem~\ref{thmGS}. The proof is complete.
\end{proof}

\subsection{The case $(\deg(f),\deg(g),\deg(c))=(1,1,1)$}\label{sec111}
In this subsection, we finish the proof of Theorem~\ref{deg111}. We need (a special case of) the following proposition, which may be of independent interest:
\begin{Prop}[= Proposition~\ref{infUCintro}]\label{infUCg}
	Let $R\in\C(z)$ be a rational map of degree $d\geq1$, and let
	\[\eta\in\partial\D=\left\lbrace z\in\C\colon\left|z\right|=1\right\rbrace.\]
	Suppose that there exists $\gamma\in\C$ such that
	\[\#\left\lbrace n\in\Z\colon R(\eta^n)\in\gamma\cdot\Z\right\rbrace=\infty.\]
	Then $\eta\in U(\C)$.
\end{Prop}
\begin{proof}
	Suppose, for the sake of contradiction, that $\eta\in\partial\D\setminus U(\C)$ but
	\[\#\left\lbrace n\in\Z\colon R(\eta^n)\in\gamma\cdot\Z\right\rbrace=\infty,\]
	for some $\gamma\in\C$. The elements in the sequence $(\eta^n)_{n\in\Z}$ are pairwise distinct. 
	
	If $\gamma=0$, then $R(\eta^n)=0$ for infinitely many $n$. Since a non-zero rational map has only finitely many roots, the sequence $(\eta^n)_{n\in\Z}$ can take at most finitely many values, which forces $\eta\in U(\C)$, a contradiction. Thus, we may assume that $\gamma\neq0$. After replacing $R$ with the rational map $R/\gamma$, we may assume that $\gamma=1$. Then we have $R(\eta^n)\in\Z$ for infinitely many $n\in\Z$.
	
	After replacing $\eta$ with $\eta^{-1}$ if necessary, we may take a sequence $(n_k)_{k=1}^\infty\subseteq\Z_{>0}$ of strictly increasing positive integers such that for every $k\geq1$, we have
	\begin{equation}\label{mkg}
		m_k:=R(\eta^{n_k})\in\Z.
	\end{equation}
	If the set $\{m_k\colon k\geq1\}$ is finite, then the quasi-finiteness of $R$ implies $\eta^{n_i}=\eta^{n_j}$ for some $i\neq j$, which forces $\eta\in U(\C)$, a contradiction. After passing to a subsequence, we may assume that the integers $m_k$ ($k\geq1$) are pairwise distinct. In particular,
	\begin{equation}\label{mkinfg}
		\lim_{k\to\infty}\left|m_k\right|=\infty.
	\end{equation}
	
	Since $\partial\D$ is compact, after passing to a subsequence, we may assume that 
	\[\lim_{k\to\infty}\eta^{n_k}=\alpha\in\partial\D.\]
	Then $\alpha\in\partial\D$ is a pole of the rational map $R$.
	
	Fix an algebraic closure $\overline{\Q}\subseteq\C$ of $\Q$ in $\C$. The standard absolute value on $\C$ is denoted by $|\cdot|$.
	
	\smallskip
	
	\textbf{Claim:} (1) $\eta\in\overline{\Q}$ and $R\in\overline{\Q}(z)$;\quad(2) $\sigma(\eta)\in\partial\D$ for every $\sigma\in\Gal(\overline{\Q}/\Q)$.
	
	\textit{Proof of the claim.} (1) Suppose that $\eta$ is transcendental over $\Q$. We can take an automorphism $\tau:\C\to\C$ of the field $\C$ such that $\eta$ and $\tau(\eta)$ are algebraically independent over $\Q$. For any $h\in\C(z)$, let $\tau(h)$ be the rational map obtained from $h$ by applying $\tau$ to all coefficients of $h$. For $k\geq1$, applying $\tau$ to \eqref{mkg} gives
	\[m_k=\tau(R)(\tau(\eta)^{n_k}).\]
	Denote the coordinates on the linear torus $\G_{m,\C}^2$ by $(x,y)$. Define $Z$ to be the curve in $\G_{m,\C}^2$ given by the equation
	\begin{equation}\label{eqZ1g}
		R(x)=\tau(R)(y).
	\end{equation}
	Let $\Gamma$ be the subgroup of $\G_m^2(\C)$ generated by the element $(\eta,\tau(\eta))$, which is of finite rank. Observe that $(\eta^{n_k},\tau(\eta)^{n_k})\in Z(\C)\cap\Gamma$ for all $k\geq1$. After passing to a subsequence, we may assume that there exists an irreducible component $Z_0$ of $Z$ such that \[(\eta^{n_k},\tau(\eta)^{n_k})\in Z_0(\C)\cap\Gamma\]
	for all $k\geq1$. Then
	\[\#\left(Z_0(\C)\cap\Gamma\right)=\infty.\]
	By the Mordell--Lang conjecture on $\G_{m,\C}^N$ proved by Laurent \cite{Laurent}, we conclude that $Z_0$ is a translate of a one-dimensional irreducible algebraic subgroup of $\G_{m,\C}^2$. Thus, $Z_0$ can be defined by an equation of the form
	\begin{equation}\label{eqZ2g}
		x^ay^b=c_1
	\end{equation}
	for some $c_1\in\C^\times$ and some $(a,b)\in\Z^2\setminus\left\lbrace(0,0)\right\rbrace$ with $\gcd(a,b)=1$. Substituting \eqref{eqZ2g} into \eqref{eqZ1g} (over a suitable finite cover of $\G_{m,\C}^2$), the fact that
	\[\deg(R)=\deg(\tau(R))=d\]
	implies that $|a/b|=1$; so $|a|=|b|=1$. Then \eqref{eqZ2g} can be rewritten as
	\begin{equation}\label{eqZ3g}
		y=cx^{\varepsilon}
	\end{equation}
	for some $c\in\C^\times$ and some $\varepsilon\in\{\pm1\}$. For every $k\geq1$, by \eqref{eqZ3g} we obtain
	\[\tau(\eta)^{n_k}=c\eta^{\varepsilon n_k}.\]
	Thus, we deduce that
	\begin{equation}\label{n2-n1g}
		\tau(\eta)^{n_2-n_1}=\eta^{\varepsilon(n_2-n_1)}.
	\end{equation}
	Recall that $n_2-n_1>0$, so \eqref{n2-n1g} contradicts the assumption that $\eta$ and $\tau(\eta)$ are algebraically independent over $\Q$. Therefore, we conclude that $\eta\in\overline{\Q}$.
	
	Since the graph of $R$ contains infinitely many $\overline{\Q}$-points $(\eta^{n_k},m_k)$ ($k\geq1$), we see that $R\in\overline{\Q}(z)$.
	
	(2) Let $\sigma\in\Gal(\overline{\Q}/\Q)$ be arbitrary. The argument in (1) applied to the curve $R(x)=\sigma(R)(y)$ shows that
	\[\sigma(\eta)^{n_2-n_1}=\eta^{\varepsilon_\sigma(n_2-n_1)}\]
	for some $\varepsilon_\sigma\in\{\pm1\}$. Taking absolute values gives $|\sigma(\eta)|=1$.
	
	The proof of the claim is complete.
	
	\smallskip
	
	We have $\alpha\in\overline{\Q}$ because $\alpha$ is a pole of $R\in\overline{\Q}(z)$. Fix a number field $K$ such that $\eta,\alpha\in K$ and $R\in K(z)$.
	
	By the claim and Kronecker's theorem, the algebraic number $\eta$ is not an algebraic integer. Then the claim and the product formula imply that we can take a non-archimedean place $v\in\sM_K$ of $K$ such that
	\begin{equation}\label{v<1g}
		0<\left|\eta\right|_v<1.
	\end{equation}
	Let $p\in\Z$ be the rational prime over which $v$ lies. Here the $v$-adic absolute value $|\cdot|_v$ is normalized by $|p|_v=p^{-1}$. Since $0<|\eta|_v<1$, we have $\eta^{n_k}\to0$ in the $v$-adic topology.
	
	If $0$ is a pole of $R$, then we obtain \[\left|m_k\right|_v=\left|R(\eta^{n_k})\right|_v\to\infty\quad\text{as}\quad k\to\infty,\]
	which contradicts the fact that $m_k\in\Z$ and hence $|m_k|_v\leq1$. Therefore, $0$ is not a pole of $R$. Set $\beta=R(0)$. Then $\beta\in K$ because $0$ is not a pole of $R$. Write
	\[R(z)-\beta=z^s\frac{P(z)}{Q(z)},\]
	where $s\geq1$ is an integer, and $P,Q\in K[z]$ are polynomials such that $P(0)\neq0$ and $Q(0)\neq0$. We deduce that there exist an integer $k_1>0$ and a number $C_1\in(0,1)$ such that for every integer $k\geq k_1$, we have
	\begin{equation}\label{vexpgrowg}
		0<\left|m_k-\beta\right|_v=A\cdot\left|\eta\right|_v^{sn_k}\leq C_1^{n_k},
	\end{equation}
	where $A=|P(0)|_v\cdot|Q(0)|_v^{-1}\in\R_{>0}$.
	
	Define
	\[g(z)=N_{K/\Q}(z-\beta)\in\Q[z],\]
	where $z$ is a variable and $N_{K/\Q}$ is the norm function for the field extension $K/\Q$. Then $g$ is a polynomial of degree $l:=[K:\Q]\geq2$, with coefficients in $\Q$. Let $D\in\Z_{>0}$ be the minimal positive integer such that
	\[D\cdot g(z)\in\Z[z].\]
	For $k\geq1$, define
	\[M_k:=D\cdot g(m_k)=D\cdot N_{K/\Q}(m_k-\beta)\in\Z.\]
	By \eqref{vexpgrowg} and the fact that $|m_k-\beta|_w$ is uniformly bounded for all places $w\vert p$ (since $m_k\in\Z$), it is easy to see that there exist an integer $k_2\geq k_1$ and a number $C_2\in(C_1^{[K_v:\Q_p]},1)$ such that for every integer $k\geq k_2$, we have $M_k\neq0$ and
	\begin{equation}\label{pexpgrow1}
		0<\left|M_k\right|_p=\left|D\right|_p\cdot\prod_{w\vert p}\left|m_k-\beta\right|_w^{[K_w:\Q_p]}\leq C_2^{n_k}.
	\end{equation}
	For every integer $k\geq k_2$, the product formula over $\Q$ gives
	\begin{equation}\label{aexpgrow1g}
		\left|M_k\right|\geq \left|M_k\right|_p^{-1}\geq C_2^{-n_k}.
	\end{equation}
	
	Since $D\cdot g$ is a polynomial of degree $l$, by \eqref{mkinfg} there exist an integer $k_3\geq k_2$ and a number $B_1\in\R_{>0}$ such that for every integer $k\geq k_3$, we have
	\begin{equation}\label{polyupg}
		\left|M_k\right|=\left|D\cdot g(m_k)\right|\leq B_1\left|m_k\right|^l. 
	\end{equation}
	Recall that $\alpha\in\partial\D$ is a pole of $R$. Let $\kappa\geq1$ be the multiplicity of the pole $\alpha$. As $k\to\infty$, since $\eta^{n_k}\to\alpha$, we have the asymptotic behavior
	\[\left|m_k\right|=\left|R(\eta^{n_k})\right|\asymp\left|\eta^{n_k}-\alpha\right|^{-\kappa}.\]
	By this asymptotic relation, \eqref{aexpgrow1g}, and \eqref{polyupg}, there exist an integer $k_4\geq k_3$, a number $B_2\in\R_{>0}$, and a number $C_3\in(C_2^{1/(\kappa l)},1)$ such that for every integer $k\geq k_4$, we have
	\begin{equation}\label{aexpgrowg}
		0<\left|\eta^{n_k}-\alpha\right|\leq B_2\left|m_k\right|^{-\frac{1}{\kappa}}\leq B_2B_1^{\frac{1}{\kappa l}}\left|M_k\right|^{-\frac{1}{\kappa l}}\leq B_2B_1^{\frac{1}{\kappa l}}C_2^{\frac{n_k}{\kappa l}}\leq C_3^{n_k}.
	\end{equation}
	
	\smallskip
	
	On the other hand, by the multiplicative form of Baker's theorem on linear forms in logarithms \cite{Baker} (see \cite[Theorem~1.11]{mulBaker}), there exist an integer $k_5\geq k_4$ and a number $\delta\in\R_{>0}$ such that for every integer $k\geq k_5$,
	\begin{equation}\label{Bakg}
		\left|\eta^{n_k}-\alpha\right|\geq n_k^{-\delta}.
	\end{equation}
	When $k\gg 1$, \eqref{Bakg} clearly contradicts \eqref{aexpgrowg}. Therefore, we conclude that $\eta$ must be a root of unity, i.e., $\eta\in U(\C)$.
\end{proof}
\begin{Rem}\label{Siegelgen}
	We are informed by Junyi Xie that Proposition~\ref{infUCg} can be easily proved by applying Siegel's theorem on integral points of curves over rings of finite type over $\Z$ (see \cite[Corollary~4.11]{Vojta21}) to the curve $y=R(x)$ in $\G_m\times\A^1$, where we view $\G_{m,\Q}=\Spec(\Q[x_1,x_2]/(x_1x_2-1))\subseteq\A^2_\Q$. We choose to keep the original proof of Proposition~\ref{infUCg} to present a different approach.
	
	Moreover, one can also obtain generalizations of Proposition~\ref{infUCg}. For example, one can replace $\gamma\Z$ with an arbitrary discrete lattice of rank two
	\[\omega_1\Z+\omega_2\Z,\]
	where $\omega_1,\omega_2\in\C$ are linearly independent over $\R$, by applying Siegel's theorem, or modifying the original proof with standard specialization arguments.
\end{Rem}

\begin{proof}[Proof of Theorem~\ref{deg111}]
	Let $f,g,c\in\C[z]$ be polynomials of degree $1$. Write
	\[f(z)=A_1z+B_1,\quad g(z)=A_2z+B_2,\quad\text{and}\quad c(z)=Cz+D,\]
	where $A_1,A_2,C\in\C^\times$ and $B_1,B_2,D\in\C$. We want to prove that $S_{f,g,c}^2$ is semi-linear.
	
	For $m\in\Z_{\geq0}$, we have
	\[f^{\circ m}(z)=A_1^mz+\beta_m,\]
	where
	\[\beta_m=B_1\frac{A_1^m-1}{A_1-1}\text{ when }A_1\neq1\quad\text{and}\quad\beta_m=B_1m\text{ when }A_1=1.\]
	For $n\in\Z_{\geq0}$, we have
	\[g^{\circ n}(z)=A_2^nz+\delta_n,\]
	where
	\[\delta_n=B_2\frac{A_2^n-1}{A_2-1}\text{ when }A_2\neq1\quad\text{and}\quad\delta_n=B_2n\text{ when }A_2=1.\]
	
	For every $(m,n)\in\Z_{\geq0}^2$, we deduce that $(m,n)\in S_{f,g,c}^2$ if and only if the following three conditions are all satisfied:
	\begin{enumerate}
		\item $A_1^m\neq C$ or $\beta_m=D$;
		\item $A_2^n\neq C$ or $\delta_n=D$;
		\item $(A_1^m-C)(\delta_n-D)=(A_2^n-C)(\beta_m-D)$.
	\end{enumerate}
	For $j\in\{1,2,3\}$, set
	\[S_j:=\left\lbrace(m,n)\in\Z_{\geq0}^2\colon\text{ the condition (}j\text{) holds}\right\rbrace.\]
	Then
	\[S_{f,g,c}^2=S_1\cap S_2\cap S_3.\]
	By Corollary~\ref{semiclosed}, it suffices to prove that $S_j$ is semi-linear for each $1\leq j\leq3$.
	
	\smallskip
	
	We first deal with the semi-linearity of $S_1$ and $S_2$. By symmetry, we only need to show the semi-linearity of $S_1$. Note that we can write
	\[S_1=\left(Z\cup W\right)\times\Z_{\geq0},\]
	where
	\[Z=\left\lbrace m\in\Z_{\geq0}\colon A_1^m\neq C\right\rbrace\quad\text{and}\quad W=\left\lbrace m\in\Z_{\geq0}\colon\beta_m=D\right\rbrace.\]
	Hence it remains to show that both $Z$ and $W$ are semi-linear by Theorem~\ref{thmGS}. It is easy to see that for all $\alpha,\gamma,\mu\in\C$, the set
	\begin{equation}\label{Psl}
		P_{\alpha,\gamma,\mu}:=\left\lbrace m\in\Z_{\geq0}\colon \alpha\mu^m=\gamma\right\rbrace\quad\text{is semi-linear};
	\end{equation}
	in fact, either $\# P_{\alpha,\gamma,\mu}\leq1$, or $P_{\alpha,\gamma,\mu}$ is an infinite arithmetic progression. (Here we set $0^0=1$.) By \eqref{Psl} and Corollary~\ref{semiclosed}, the set
	\[Z=\Z_{\geq0}\setminus P_{1,C,A_1}\]
	is semi-linear. If $A_1=1$ and $B_1=0$, then we have $W=\emptyset$ when $D\neq0$, and $W=\Z_{\geq0}$ when $D=0$. If $A_1=1$ and $B_1\neq0$, then $W=\left\lbrace D/B_1\right\rbrace\cap\Z_{\geq0}$ which has cardinality $\leq1$. We have shown that $W$ is semi-linear if $A_1=1$. If $A_1\neq1$, then $W$ is semi-linear by \eqref{Psl}. Thus, both $Z$ and $W$ are semi-linear, and so is $S_1$.
	
	\smallskip
	
	Now we consider the semi-linearity of $S_3$. The analysis for $S_3$ is divided into three cases.
	
	\smallskip
	
	\textbf{Case 1.} Suppose that $A_1=A_2=1$.
	
	If $C=1$, then we have $S_3=\Z_{\geq0}^2$ which is semi-linear. Suppose now that $C\neq1$. Then
	\[S_3=\left\lbrace (m,n)\in\Z_{\geq0}^2\colon mB_1=nB_2\right\rbrace.\]
	
	If $B_1=0$, then $S_3=\Z_{\geq0}\times\{0\}$ when $B_2\neq0$, and $S_3=\Z_{\geq0}^2$ when $B_2=0$. We see that $S_3$ is semi-linear if $B_1=0$. Similarly, $S_3$ is semi-linear if $B_2=0$.
	
	Suppose now that $B_1\neq0$ and $B_2\neq0$. If $B_2/B_1\notin\Q_{>0}$, then it is easy to see that $S_3=\left\lbrace(0,0)\right\rbrace$ is a singleton. Suppose now that $B_2/B_1\in\Q_{>0}$, and write
	\[\frac{B_2}{B_1}=\frac{p}{q},\]
	where $p$ and $q$ are coprime positive integers. Then
	\[S_3=\left\lbrace(m,n)\in\Z_{\geq0}^2\colon qm=pn\right\rbrace,\]
	which is definable in the Presburger arithmetic $\Pr$, and hence semi-linear by Theorem~\ref{thmGS}.
	
	\smallskip
	
	\textbf{Case 2.} Suppose that exactly one of $A_1$ and $A_2$ is equal to $1$.
	
	Without loss of generality, we may assume that $A_1=1$ and $A_2\neq1$. Set
	\[E=A_2D-B_2C+B_2-D\in\C.\]
	
	Suppose first that $B_1=0$.
	One computes that
	\[S_3=P_{E,E,A_2}\]
	is semi-linear by \eqref{Psl}.
	
	Suppose now that $B_1\neq0$. Note that the condition (3) for $(m,n)\in\Z_{\geq0}^2$ is equivalent to:
	\[mB_1(A_2^n-C)=D(A_2^n-C)+(1-C)(\delta_n-D).\]
	Define
	\[R(z)=\frac{E(z-1)}{B_1(A_2-1)(z-C)}\in\C(z).\]
	We can decompose $S_3$ as
	\[S_3=S_{31}\sqcup S_{32},\]
	where
	\[S_{31}=\left\lbrace(m,n)\in\Z_{\geq0}^2\colon A_2^n=C\text{ and }(1-C)(\delta_n-D)=0\right\rbrace\]
	and
	\[S_{32}=\left\lbrace(m,n)\in\Z_{\geq0}^2\colon A_2^n\neq C\text{ and }m=R(A_2^n)\right\rbrace.\]
	By \eqref{Psl} and Theorem~\ref{thmGS}, we deduce that
	\[S_{31}=P_{1,C,A_2}\cap P_{B_2(1-C),(1-C)(A_2D+B_2-D),A_2}\]
	is semi-linear. Observe that $S_{32}=S_{33}\setminus P_{1,C,A_2}$, where
	\[S_{33}:=\left\lbrace(m,n)\in\Z_{\geq0}^2\colon m=R(A_2^n)\right\rbrace.\]
	Since both $S_{31}$ and $P_{1,C,A_2}$ are semi-linear, it remains to show that $S_{33}$ is semi-linear. If $R\in\C$ is a constant, then
	\[S_{33}=\left(\left\lbrace R\right\rbrace\cap\Z_{\geq0}\right)\times\Z_{\geq0}\]
	which is semi-linear. We may assume now that $\deg(R)=1$. We consider the following two cases (2.1) and (2.2) to show the semi-linearity of $S_{33}$.
	
	\textbf{(2.1)} Suppose that $|A_2|\neq1$.
	
	Define
	\[R_\infty:=\frac{E}{B_1(A_2-1)}\in\C\text{ if }\left|A_2\right|>1,\,\text{ and }\,R_\infty:=\frac{E}{CB_1(A_2-1)}\in\C\text{ if }\left|A_2\right|<1.\]
	Hence we have
	\[\lim_{n\to\infty}R(A_2^n)=R_\infty\in\C.\]
	Note that $R(A_2^n)\neq R_{\infty}$ for all $n\in\Z_{\geq0}$ because $\deg(R)=1$.
	We conclude that
	\[\#\left\lbrace n\in\Z_{\geq0}\colon R(A_2^n)\in\Z_{\geq0}\right\rbrace<\infty.\]
	Therefore, $S_{33}$ is also finite, hence semi-linear.
	
	\textbf{(2.2)} Suppose that $|A_2|=1$.
	
	If
	\[\#\left\lbrace n\in\Z_{\geq0}\colon R(A_2^n)\in\Z_{\geq0}\right\rbrace<\infty,\]
	then $S_{33}$ is finite and hence semi-linear. We may assume now that
	\[\#\left\lbrace n\in\Z_{\geq0}\colon R(A_2^n)\in\Z_{\geq0}\right\rbrace=\infty.\]
	A special case of Proposition~\ref{infUCg} gives
	\[A_2\in U(\C)\setminus\{1\}.\]
	Let $s\in\Z_{\geq2}$ be the order of $A_2$. Then
	\[S_{33}=\bigsqcup_{j=0}^{s-1}\left(\left(\left\lbrace R(A_2^j)\right\rbrace\cap\Z_{\geq0}\right)\times\left(j+s\Z_{\geq0}\right)\right),\]
	which is semi-linear.
	
	\smallskip
	
	\textbf{Case 3.} Suppose that $A_1\neq1$ and $A_2\neq1$.
	
	Define
	\[U:=\frac{B_1}{A_1-1}\quad\text{and}\quad V:=\frac{B_2}{A_2-1}.\]
	Define $Z$ to be the closed subscheme of the algebraic torus $\G_{m,\C}^2$ given by the equation
	\[(U-V)xy+(D+V-CU)x-(D+U-CV)y+C(U-V)=0,\]
	where $(x,y)$ are the coordinates on $\G_{m,\C}^2$. For every $(m,n)\in\Z_{\geq0}^2$, we have
	\[(m,n)\in S_3\iff (A_1^m,A_2^n)\in Z(\C).\]

	Let $\Gamma$ be the subgroup of $\G_m^2(\C)$ generated by $\left\lbrace(A_1,1),(1,A_2)\right\rbrace$, which is of finite rank. By the Mordell--Lang conjecture on $\G_{m,\C}^N$ proved by Laurent \cite{Laurent}, we conclude that
	\[Z(\C)\cap\Gamma=\bigcup_{j=1}^{k}g_j\cdot H_j(\C),\]
	where $k\in\Z_{\geq0}$, and for each $1\leq j\leq k$, $H_j$ is an irreducible algebraic subgroup of $\G_m^2$ and $g_j\in\G_m^2(\C)$. For $1\leq j\leq k$, define
	\[S_{3j}=\left\lbrace(m,n)\in\Z_{\geq0}^2\colon (A_1^m,A_2^n)\in g_j\cdot H_j(\C)\right\rbrace.\]
	Then
	\[S_3=\bigcup_{j=1}^k S_{3j}.\]
	It suffices to prove that $S_{3j}$ is semi-linear for all $1\leq j\leq k$.
	
	Let $j\in\{1,\dots,k\}$ be arbitrary. We may assume that $S_{3j}\neq\emptyset$. Define
	\[S_j^{\mathrm{full}}=\left\lbrace (m,n)\in\Z^2\colon (A_1^m,A_2^n)\in H_j(\C)\right\rbrace\]
	and
	\[S_j^{\mathrm{full},t}=\left\lbrace (m,n)\in\Z^2\colon (A_1^m,A_2^n)\in g_j\cdot H_j(\C)\right\rbrace.\]
	Take $(m_1,n_1)\in S_{3j}\subseteq S_j^{\mathrm{full},t}$. It is clear that
	\begin{equation}\label{fullt}
		S_j^{\mathrm{full},t}=S_j^{\mathrm{full}}+(m_1,n_1).
	\end{equation}
	Note that $S_j^{\mathrm{full}}$ is a subgroup of $(\Z^2,+)$, hence a free abelian group of rank $\leq2$. Set $r=\mathrm{rank}\left(S_j^{\mathrm{full}}\right)\in\{0,1,2\}$. Then we can write
	\begin{equation}\label{full}
		S_j^{\mathrm{full}}=\left\lbrace \sum_{i=1}^r a_iv_i\colon a_i\in\Z\text{ for }1\leq i\leq r\right\rbrace
	\end{equation}
	for some $\Q$-linearly independent vectors $(v_i)_{1\leq i\leq r}\subseteq \Z^2$. Here $S_j^{\mathrm{full}}=\left\lbrace(0,0)\right\rbrace$ if $r=0$. The descriptions \eqref{full} and \eqref{fullt} show that both $S_j^{\mathrm{full}}$ and $S_j^{\mathrm{full},t}$ are definable in the structure $(\Z,+,0,1)$. It follows routinely that
	\[S_{3j}=S_j^{\mathrm{full},t}\cap\Z_{\geq0}^2\]
	is definable in $\Pr=\left(\Z_{\geq0},+,<,0,1,(P_n)_{n\in\Z_{\geq2}}\right)$. Thus, $S_{3j}$ is semi-linear by Theorem~\ref{thmGS}.
	
	\medskip
	
	The proof is complete.
\end{proof}

\begin{Rem}
	In fact, (a small modification of) the proof of Theorem~\ref{deg111} shows that for all polynomials $f,g,c\in\C[z]$ of degree $1$, the \emph{full rank-two recurrence set}
	\[S_{f,g,c}^{2,\mathrm{full}}:=\left\lbrace(m,n)\in\Z^2\colon\exists\la\in\C,f^{\circ m}(\la)=g^{\circ n}(\la)=c(\la)\right\rbrace\]
	is definable in the structure $(\Z,+,0,1)$.
\end{Rem}

\subsection{Partial results for rational maps}\label{secratpf}
\begin{proof}[Proof of Theorem~\ref{thmratg}]
Let $f,g,c\in\C(z)$ be rational maps such that $\deg(f)\geq2$ and $\deg(g)\geq2$.

By Corollary~\ref{rat1to2}, it suffices to prove (\ref{rat1g}) and (\ref{rat3g}) of Theorem~\ref{thmratg}.

\medskip

\textbf{(1)} Suppose that $\deg(f)=\deg(g)$ and
\begin{equation}\label{lmfg}
	f^{\circ (l+k)}=f^{\circ l}\circ g^{\circ k}
\end{equation}
for some integers $l\geq0$ and $k\geq1$. We prove that the rank-one recurrence set $\hat{S}_{f,g,c}$ is semi-linear.

We may assume that $(l,k)\in\Z_{\geq0}\times\Z_{>0}$ is minimal in the sense that for all integers $l^\prime\geq0$ and $k^\prime\geq1$, if $l^\prime<l$, or $l^\prime=l$ and $k^\prime<k$, then
\[f^{\circ (l^\prime+k^\prime)}\neq f^{\circ l^\prime}\circ g^{\circ k^\prime}.\]

For $j\in\{l,l+1,\dots,l+k-1\}$, define
\[S_j=(j+k\Z_{\geq0})\cap\hat{S}_{f,g,c},\quad\text{where}\quad j+k\Z_{\geq0}=\left\lbrace j+km\colon m\in\Z_{\geq0}\right\rbrace.\]
Set $S_s:=\left[0,l\right)\cap\hat{S}_{f,g,c}$, which is finite, hence semi-linear. Then we obtain the decomposition
\[\hat{S}_{f,g,c}=S_s\sqcup\bigsqcup_{l\leq j\leq l+k-1}S_j.\]
Thus, it suffices to show that $S_j$ is semi-linear for $l\leq j\leq l+k-1$.

Fix an integer $l\leq j\leq l+k-1$. We consider the following two cases.

\smallskip

(i) Suppose that $j=0$. Then $l=0$ and $f^{\circ k}=g^{\circ k}$. Let $m\in\Z_{\geq0}$. If
\[\deg(g)^{km}\neq\deg(c),\]
then the equation
\[f^{\circ km}(\la)=g^{\circ km}(\la)=c(\la)\quad(\la\in\P^1(\C))\]
always has a solution in $\P^1(\C)$ because $f^{\circ km}=g^{\circ km}$. We conclude that
\[\#\left(k\Z_{\geq0}\setminus S_0\right)\leq1;\]
hence, $S_0\subseteq k\Z_{\geq0}$ must be semi-linear.

\smallskip

(ii) Suppose that $j>0$. The minimality of $(l,k)$ implies that $f^{\circ j}\neq g^{\circ j}$ in $\C(z)$. Define 
\[J_j=\left\lbrace e\in\P^1(\C)\colon f^{\circ j}(e)=g^{\circ j}(e)\right\rbrace,\]
which is a non-empty finite set. For $m\in\Z_{\geq0}$, by \eqref{lmfg} and the assumption that $j\geq l$, we obtain
\begin{align*}
	&j+km\in S_j\text{, i.e., }\exists\la\in\P^1(\C),\,f^{\circ (j+km)}(\la)=g^{\circ (j+km)}(\la)=c(\la)\\
	\iff&\exists\la\in\P^1(\C),\,f^{\circ j}(g^{\circ km}(\la))=g^{\circ j}(g^{\circ km}(\la))=c(\la)\\
	\iff&\exists e\in J_j,\,\exists\la\in\P^1(\C),\,g^{\circ km}(\la)=e\text{ and }g^{\circ j}(e)=c(\la).
\end{align*}
We deal with two subcases as follows.

\smallskip

(ii.i) Suppose that $j>0$ and $c\in\C$. Set
\[J_j^\prime=\left\lbrace e\in J_j\colon g^{\circ j}(e)=c\right\rbrace,\]
which is a finite set (it may be empty). Note that for every $m\in\Z_{\geq0}$ and $e\in J_j^\prime$, the equation \[g^{\circ km}(\la)=e\quad(\la\in\P^1(\C))\]
always has a solution in $\P^1(\C)$. Therefore, $S_j=j+k\Z_{\geq0}$ is an arithmetic progression if $J_j^\prime\neq\emptyset$, and $S_j=\emptyset$ if $J_j^\prime=\emptyset$. In both cases, $S_j$ is semi-linear.

\smallskip

(ii.ii) Suppose that $j>0$ and $\deg(c)\geq1$. For every $e\in J_j$, define
\[T_j^e=\left\lbrace \la\in\P^1(\C)\colon c(\la)=g^{\circ j}(e)\right\rbrace,\]
which is a non-empty finite set. Let $e\in J_j$ and $\la\in T_j^e$. Set
\[S_{j,e,\la}=\left\lbrace j+km\colon m\in\Z_{\geq0}\text{ such that }g^{\circ km}(\la)=e\right\rbrace.\]
If $e\notin O_{g^{\circ k}}(\la)$, then $S_{j,e,\la}=\emptyset$. If $e\in O_{g^{\circ k}}(\la)$, then it is easy to see that $S_{j,e,\la}$ is an arithmetic progression by analyzing the $g^{\circ k}$-preperiodicity of $\la$. In both cases, $S_{j,e,\la}$ is semi-linear. Therefore, we conclude that
\[S_j=\bigcup_{e\in J_j}\bigcup_{\la\in T_j^e} S_{j,e,\la}\]
is also semi-linear.

\smallskip

According to the above case-by-case analysis, $S_j$ is semi-linear for $l\leq j\leq l+k-1$, and the proof of Theorem~\ref{thmratg}~(\ref{rat1g}) is finished.

\medskip 

	\textbf{(3)} Suppose that $c\in\C$ and $f$ is non-exceptional. We want to prove that the rank-two recurrence set $\hat{S}_{f,g,c}^2$ is semi-linear.
	
	Set $d_1=\deg(f)\geq2$ and $d_2=\deg(g)\geq2$.
	
	By Remark~\ref{Q2ratfinite}, we may assume that the set
	\[\hat{Q}_{f,g,c}^2=\left\lbrace\la\in\P^1(\C)\colon\exists(m,n)\in\Z_{\geq0}^2,f^{\circ m}(\la)=g^{\circ n}(\la)=c(\la),f^{\circ m}\neq c,g^{\circ n}\neq c\right\rbrace\]
	is infinite, and that we have
	\[P:=\PrePer(f,\C)=\PrePer(g,\C),\quad\mu:=\mu_f=\mu_g,\quad\text{and}\quad J:=J(f)=J(g).\]
	Then $g$ is also non-exceptional by \cite{Zdunik90} because $\mu_g=\mu_f$. Applying Theorem~\ref{ratsamemu} to $(f,g)$, we obtain
	\begin{equation}\label{fglkg}
		f^{\circ 2l}=f^{\circ l}\circ g^{\circ k}\quad\text{and}\quad g^{\circ 2k}=g^{\circ k}\circ f^{\circ l}
	\end{equation}
	for some integers $l,k\geq1$. Comparing the degrees gives $d_1^l=d_2^k$.
	
	By a rational map version of Lemma~\ref{lemSs}, it is sufficient to prove the semi-linearity of the set
	\[\hat{S}^\prime:=\left\lbrace(m,n)\in\hat{S}_{f,g,c}^2\colon m\geq l\text{ and }n\geq k\right\rbrace.\]
	Define
	\[C_1=\left\lbrace(m,n)\in\Z_{\geq0}^2\colon m\geq l,\,n\geq k,\,km\geq ln\right\rbrace\]
	and
	\[C_2=\left\lbrace(m,n)\in\Z_{\geq0}^2\colon m\geq l,\,n\geq k,\,km<ln\right\rbrace,\]
	which are both definable in $\Pr$. Set
	\[\hat{S}_1=\hat{S}_{f,g,c}^2\cap C_1\quad\text{and}\quad\hat{S}_2=\hat{S}_{f,g,c}^2\cap C_2.\]
	Then $\hat{S}^\prime=\hat{S}_1\sqcup\hat{S}_2$. It remains to show the semi-linearity of $\hat{S}_1$ and $\hat{S}_2$. We will show that $\hat{S}_1$ is semi-linear; the case for $\hat{S}_2$ is symmetric.
	
	\smallskip
	
	For every $i\in\{k,k+1,\dots,2k-1\}$, define $C_{1i}=\{(m,n)\in C_1\colon n\in k\Z_{\geq0}+i\}$, which is definable in $\Pr$. Then we have decompositions
	\[C_1=\bigsqcup_{i=k}^{2k-1}C_{1i}\quad\text{and}\quad\hat{S}_1=\bigsqcup_{i=k}^{2k-1}\hat{S}_{1i},\]
	where $\hat{S}_{1i}=\hat{S}_{f,g,c}^2\cap C_{1i}$ for $k\leq i\leq 2k-1$.
	
	Let $i\in\{k,k+1,\dots,2k-1\}$ be arbitrary. Define
	\[Y_i:=(g^{\circ i})^{-1}(c),\]
	which is a non-empty finite subset of $\P^1(\C)$. Note that for every $(m,n)\in C_{1i}$, we have \[q=q(n):=\frac{n-i}{k}\in\Z_{\geq0}\quad\text{and}\quad m-lq\geq l(\frac{n}{k}-q)=l\frac{i}{k}\geq l.\]
	For every $(m,n)\in C_{1i}$, by applying \eqref{fglkg} $q$ times we obtain
	\begin{align*}
		&(m,n)\in\hat{S}_{1i}\text{, i.e., }\exists\la\in\P^1(\C),\,f^{\circ m}(\la)=g^{\circ n}(\la)=c\\
		\iff&\exists\la\in\P^1(\C),\,f^{\circ (m-lq)}\circ g^{\circ kq}(\la)=g^{\circ i}\circ g^{\circ kq}(\la)=c\\
		\iff&\exists\la\in\P^1(\C),\,\exists y\in Y_i,\,f^{\circ (m-lq)}(y)=c\text{ and }g^{\circ kq}(\la)=y\\
		\iff&\exists y\in Y_i,\,f^{\circ (m-lq)}(y)=c,
	\end{align*}
	where the last equivalence follows from the fact that the equation $g^{\circ kq}(\la)=y$ ($\la\in\P^1(\C)$) always has a solution in $\P^1(\C)$. For all $h,\phi\in\C(z)$ and all $w\in\P^1(\C)$, as in \eqref{Shla}, it is easy to see that the set
	\[\hat{S}_{h,\phi}(w):=\left\lbrace s\in\Z_{\geq0}\colon f^{\circ s}(w)=\phi(w)\right\rbrace\]
	is semi-linear, and hence is definable in $\Pr$. Then we conclude that
	\begin{align*}
		&\hat{S}_{1i}=\left\lbrace (m,n)\in C_{1i}\colon\exists y\in Y_i,\,m-lq\in\hat{S}_{f,c}(y)\right\rbrace\\
		=&\left\lbrace (m,n)\in C_{1i}\colon\exists y\in Y_i,\,km-ln+li\in k\hat{S}_{f,c}(y)\right\rbrace
	\end{align*}
	is definable in $\Pr$, since $Y_i$ is finite and the function $(m,n)\mapsto km-ln+li$ is definable in $\Pr$. By Theorem~\ref{thmGS}, $\hat{S}_{1i}$ is semi-linear; hence, $\hat{S}_1=\bigsqcup_{i=k}^{2k-1}\hat{S}_{1i}$ is also semi-linear. The proof of Theorem~\ref{thmratg}~(\ref{rat3g}) is finished.
\end{proof}

\subsection*{Acknowledgment}
The author would like to thank She Yang and Junyi Xie for useful discussions.


\begin{thebibliography}{BHMV94}
	\bibitem[Bak75]{Baker}
	Alan~Baker.
	\newblock A sharpening of the bounds for linear forms in logarithms \Rom{3}.
	\newblock {\em Acta Arith.}, 27:247--252, 1975.
	
	\bibitem[Bak09]{Bak09}
	Matthew~Baker.
	\newblock A finiteness theorem for canonical heights attached to rational maps over function fields.
	\newblock {\em J. Reine Angew. Math.}, 626:205--233, 2009.
	
	\bibitem[BD11]{BD11}
	Matthew~Baker and Laura~{D}e{M}arco.
	\newblock Preperiodic points and unlikely intersections.
	\newblock {\em Duke Math. J.}, 159(1):1--29, 2011.
	
	\bibitem[BGT16]{DMLbook}
	Jason~P.~Bell, Dragos~Ghioca, and Thomas~J.~Tucker.
	\newblock {\em The dynamical {M}ordell--{L}ang conjecture}, {\em Math. Surveys Monogr.}, vol.~210.
	\newblock Amer. Math. Soc., Providence, RI, 2016.
	
	\bibitem[BHPT24]{Tits}
	Jason~P.~Bell, Keping~Huang, Wayne~Peng, and Thomas~J.~Tucker.
	\newblock A Tits alternative for endomorphisms of the projective line.
	\newblock {\em J. Eur. Math. Soc. (JEMS)}, 26(12):4903--4922, 2024.
	
	\bibitem[BG06]{HtBook}
	Enrico Bombieri and Walter Gubler.
	\newblock {\em Heights in {D}iophantine geometry}, {\em New Math. Monogr.}, vol.~4.
	\newblock Cambridge Univ. Press, Cambridge, 2006.
	
	\bibitem[BHMV94]{Buchi}
	V{\'e}ronique~Bruy{\`e}re, Georges~Hansel, Christian~Michaux, and Roger~Villemaire.
	\newblock Logic and {$p$}-recognizable sets of integers.
	\newblock {\em Bull. Belg. Math. Soc. Simon Stevin}, 1(2):191--238, 1994.
	
	\bibitem[Bug18]{mulBaker}
	Yann~Bugeaud.
	\newblock {\em Linear forms in logarithms and applications}, {\em IRMA Lect. Math. Theor. Phys.}, vol.~28.
	\newblock Eur. Math. Soc., Z{\"u}rich, 2018.
	
	\bibitem[BCZ03]{BCZgcd}
	Yann~Bugeaud, Pietro~Corvaja, and Umberto~Zannier.
	\newblock An upper bound for the {G.C.D.} of {$a^n-1$} and {$b^n-1$}.
	\newblock {\em Math. Z.}, 243(1):79--84, 2003.
	
	\bibitem[CS93]{CallSilverman}
	Gregory~S.~Call and Joseph~H.~Silverman.
	\newblock Canonical heights on varieties with morphisms.
	\newblock {\em Compos. Math.}, 89:163--205, 1993.
	
	\bibitem[FG22]{FG22}
	Charles~Favre and Thomas~Gauthier.
	\newblock {\em The arithmetic of polynomial dynamical pairs}, {\em Ann. of Math. Stud.}, vol.~214.
	\newblock Princeton Univ. Press, Princeton, NJ, 2022.
	
	\bibitem[FLM83]{FLM83}
	Alexandre~Freire, Artur~Lopes, and Ricardo~Ma{\~{n}}{\'e}.
	\newblock An invariant measure for rational maps.
	\newblock {\em Bol. Soc. Bras. Mat.}, 14(1):45--62, 1983.
	
	\bibitem[GNY19]{GNY19}
	Dragos~Ghioca, Khoa~D.~Nguyen, and Hexi~Ye.
	\newblock The dynamical {M}anin--{M}umford conjecture and the dynamical {B}ogomolov conjecture for split rational maps.
	\newblock {\em J. Eur. Math. Soc. (JEMS)}, 21(5):1571--1594, 2019.
	
	\bibitem[GT09]{GT09}
	Dragos~Ghioca and Thomas~J.~Tucker.
	\newblock Periodic points, linearizing maps, and the dynamical {M}ordell--{L}ang problem.
	\newblock {\em J. Number Theory}, 129(6):1392--1403, 2009.
	
	\bibitem[GS66]{GS66}
	Seymour~Ginsburg and Edwin~H.~Spanier.
	\newblock Semigroups, Presburger formulas, and languages.
	\newblock {\em Pacific J. Math.}, 16(2):285--296, 1966.
	
	\bibitem[HT17]{HT17}
	Liang-Chung~Hsia and Thomas~J.~Tucker.
	\newblock Greatest common divisors of iterates of polynomials.
	\newblock {\em Algebra Number Theory}, 11(6):1437--1459, 2017.
	
	\bibitem[Lau84]{Laurent}
	Michel Laurent.
	\newblock Equations diophantiennes exponentielles.
	\newblock {\em Invent. Math.}, 78(2):299--327, 1984.
	
	\bibitem[LP97]{LP97}
	Genadi~Levin and Feliks~Przytycki.
	\newblock When do two rational functions have the same Julia set?
	\newblock {\em Proc. Amer. Math. Soc.}, 125(7):2179--2190, 1997.
	
	\bibitem[Lyu83]{Lyubich83}
	Mikhail~Ju.~Lyubich.
	\newblock Entropy properties of rational endomorphisms of the {R}iemann sphere.
	\newblock {\em Ergodic Theory Dynam. Systems}, 3:351--385, 1983.
	
	\bibitem[Ma{\~{n}}83]{Mane83}
	Ricardo~Ma{\~{n}}{\'e}.
	\newblock On the uniqueness of the maximizing measure for rational maps.
	\newblock {\em Bol. Soc. Bras. Mat.}, 14(1):27--43, 1983. 
	
	\bibitem[Mar02]{217}
	David~Marker.
	\newblock {\em Model Theory: An Introduction}, {\em Grad. Texts in Math.}, vol.~217.
	\newblock Springer, New York, NY, 2002.
	
	\bibitem[Mil06]{milnor2006lattes}
	John Milnor.
	\newblock On {L}att{\`e}s maps.
	\newblock In {\em Dynamics on the {Riemann} sphere: {A} {B}odil {B}ranner {F}estschrift}, pages 9--43. Eur. Math. Soc., Z{\"u}rich, 2006.

	\bibitem[NZ25]{NZ25}
	Chatchai~Noytaptim and Xiao~Zhong.
	\newblock A finiteness result for common zeros of iterates of rational functions.
	\newblock {\em Int. Math. Res. Not. IMRN}, 2025(11), article ID rnaf149, 2025.
	
	\bibitem[Pak20]{Paksamemu}
	Fedor~Pakovich.
	\newblock On rational functions sharing the measure of maximal entropy.
	\newblock {\em Arnold Math. J.}, 6:387--396, 2020.
	
	\bibitem[Par11]{LTE}
	Amir~Hossein~Parvardi.
	\newblock Lifting the exponent lemma ({LTE}).
	\newblock 2011. Available at \url{https://pregatirematematicaolimpiadejuniori.wordpress.com/wp-content/uploads/2016/07/lte.pdf}.
	
	\bibitem[SS95]{SS95}
	Walter~Schmidt and Norbert~Steinmetz.
	\newblock The polynomials associated with a {J}ulia set.
	\newblock {\em Bull. Lond. Math. Soc.}, 27(3):239--241, 1995.
	
	\bibitem[Sil07]{Silverman2007}
	Joseph~H.~Silverman.
	\newblock {\em The arithmetic of dynamical systems}, {\em Grad. Texts in Math.}, vol.~241.
	\newblock Springer, New York, NY, 2007.
	
	\bibitem[Voj21]{Vojta21}
	Paul~Vojta.
	\newblock Roth's theorem over arithmetic function fields.
	\newblock {\em Algebra Number Theory}, 15(8):1943--2017, 2021.
	
	\bibitem[Xie26]{xDML}
	Junyi~Xie.
	\newblock Around the dynamical {M}ordell--{L}ang conjecture.
	\newblock In {\em Algebraic, Complex, and Arithmetic Dynamics}, {\em Simons Symp.},
	pages 59--98. Springer, Cham, 2026.
	
	\bibitem[YZ26]{YZ26}
	She~Yang and Xiao~Zhong.
	\newblock Dynamical {GCD} problems and a variant of the dynamical {M}ordell--{L}ang conjecture.
	\newblock arXiv:2602.18302, 2026.
	
	\bibitem[Ye15]{Ye}
	Hexi~Ye.
	\newblock Rational functions with identical measure of maximal entropy.
	\newblock {\em Adv. Math.}, 268:373--395, 2015.
	
	\bibitem[YZ21]{YZ21}
	Xinyi~Yuan and Shou-Wu~Zhang.
	\newblock The arithmetic {H}odge index theorem for adelic line bundles \Rom{2}: finitely generated fields.
	\newblock arXiv:1304.3539v2, 2021.
	
	\bibitem[Zan12]{UnlikelyBook}
	Umberto~Zannier.
	\newblock {\em Some problems of unlikely intersections in arithmetic and geometry}, {\em Ann. of Math. Stud.}, vol.~181.
	\newblock Princeton Univ. Press, Princeton, NJ, 2012.
	\newblock With appendices by David~Masser.
	
	\bibitem[Zdu90]{Zdunik90}
	Anna~Zdunik.
	\newblock Parabolic orbifolds and the dimension of the maximal measure for rational maps.
	\newblock {\em Invent. Math.}, 99(3):627--649, 1990.
\end{thebibliography}
\end{document}